\documentclass[11pt]{amsart}
\usepackage{amssymb,amsmath}

\theoremstyle{plain}
\newtheorem{thrm}{Theorem}[section]
\newtheorem{lemma}[thrm]{Lemma}
\newtheorem{prop}[thrm]{Proposition}
\newtheorem{cor}[thrm]{Corollary}
\newtheorem{rmrk}[thrm]{Remark}
\newtheorem{dfn}[thrm]{Definition}

\numberwithin{equation}{section}

\begin{document}
\newcommand{\SL}{\mathcal L^{1,p}( D)}
\newcommand{\Lp}{L^p( Dega)}
\newcommand{\CO}{C^\infty_0( \Omega)}
\newcommand{\Rn}{\mathbb R^n}
\newcommand{\Rm}{\mathbb R^m}
\newcommand{\R}{\mathbb R}
\newcommand{\Om}{\Omega}
\newcommand{\Hn}{\mathbb H^n}
\newcommand{\N}{\mathbb N}
\newcommand{\aB}{\alpha B}
\newcommand{\eps}{\epsilon}
\newcommand{\BVX}{BV_X(\Omega)}
\newcommand{\p}{\partial}
\newcommand{\IO}{\int_\Omega}
\newcommand{\bG}{\mathbb{G}}
\newcommand{\bg}{\mathfrak g}
\newcommand{\Bux}{\mbox{Box}}
\newcommand{\al}{\alpha}
\newcommand{\til}{\tilde}
\newcommand{\nuX}{\boldsymbol{\nu}^X}
\newcommand{\bN}{\boldsymbol{N}}
\newcommand{\nh}{\nabla_H}
\newcommand{\deh}{\Delta_H}
\newcommand{\rh}{|\nabla_H \rho|^2}
\newcommand{\uh}{|\nabla_H u|^2}
\newcommand{\Gp}{G_{D,p}}
\newcommand{\n}{\boldsymbol \nu}
\newcommand{\ve}{\varepsilon}
\newcommand{\dsh}{|\nabla_H \rho| d\sigma_H}
\newcommand{\la}{\lambda}
\newcommand{\vf}{\varphi}
\newcommand{\rhh}{|\nabla_H \rho|}
\newcommand{\Ba}{\mathcal{B}_\alpha}
\newcommand{\Za}{Z_\alpha}
\newcommand{\ra}{\rho_\alpha}
\newcommand{\na}{\nabla_\alpha}
\newcommand{\vt}{\vartheta}


\title[Frequency of Harmonic functions in Carnot groups, etc.]
{Frequency of Harmonic functions in Carnot groups and for operators of Baouendi type}

\dedicatory{Dedicated to Ermanno Lanconelli, on the occasion of his birthday}

\author{Nicola Garofalo}
\address{Dipartimento di Ingegneria Civile, Edile e Ambientale (DICEA) \\ Universit\`a di Padova\\ 35131 Padova, ITALY}
 
 \address{Department of Mathematics\\
Purdue University \\
West Lafayette IN 47907-1968}

\email[Nicola
Garofalo]{rembrandt54@gmail.com}

\thanks{First author supported in part by NSF Grant DMS-1001317}

\author{Kevin Rotz}
\address{Department of Mathematics \\
Purdue University \\
West Lafayette IN 47907-1968} \email[Kevin Rotz]{kevin.rotz@gmail.com}
%

\date{\today}

%
%

\maketitle

\tableofcontents

\begin{abstract}
We introduce a new notion of \emph{Almgren's frequency} which is adapted to solutions of a sub-Laplacian (harmonic functions) on a Carnot group of arbitrary step $\bG$. With this notion we investigate some new functionals associated with the frequency, and obtain  monotonicity formulas for the relevant harmonic functions, or for the solutions of a closely connected class of degenerate second order operators of Baouendi type, see \eqref{Ba} below. The results proved in this paper provide some new insight into the deep link existing between the growth properties of the frequency, and the local and global structure of the relevant harmonic functions in these non-elliptic, or \emph{subelliptic}, settings. 

\end{abstract}

\section{Introduction and statement of the results}\label{S:intro}

The celebrated \emph{frequency function} of Almgren, and its monotonicity properties, play a fundamental role in several questions in partial differential equations and geometric measure theory. We recall that in \cite{A} Almgren proved that if $u$ is a harmonic function in, say, the closed unit ball $\overline B_1 \subset \Rn$, then its frequency
\[
r \to N(u,r) = \frac{r \int_{B_r} |\nabla u|^2}{\int_{S_r} u^2},\ \ \ \ \ \ 0<r<1,
\]
where $S_r = \p B_r$, is non-decreasing on the interval $(0,1)$. In particular, $N(u,\cdot)\in L^\infty(0,1)$, with $||N(u,\cdot)||_{L^\infty(0,1)} = N(u,1)$. This latter fact implies the following doubling inequality
\[
\int_{B_{2r}} u^2 \le C(n,||u||_{W^{1,2}(B_1)}) \int_{B_{r}} u^2,\ \ \ \ 0<r<1/2,
\]
which, in turn, proves the strong unique continuation property of harmonic functions. For this aspect the reader should see  the papers \cite{GLin1}, \cite{GLin2}, by F.H. Lin and the first named author. In those papers Almgren's monotonicity was generalized to elliptic operators in divergence form with Lipschitz continuous coefficients, and the unique continuation property was proved.

The recent work of Caffarelli, Salsa and Silvestre \cite{CSS}  has brought to light a fundamental new role of Almgren's monotonicity of the frequency in the study of free boundaries for lower-dimensional obstacle problems. The subsequent work of Petrosyan and the first named author \cite{GP} on the analysis of the singular points in the lower-dimensional obstacle problem has underscored a deep connection between the monotonicity of the frequency and two one-parameter families of new monotonicity formulas inspired to the single ones, originally discovered by G. Weiss \cite{W}, and R. Monneau \cite{M}, for the classical obstacle problem. 
One remarkable new aspect which was brought to light in \cite{CSS} is the connection between free boundary problems for the fractional Laplacian and monotonicity formulas of Almgren type for some degenerate elliptic operators of subelliptic type which belong to a class originally introduced by Baouendi in \cite{B}. 

It was this connection that led us to introduce, in the present paper, a new notion of \emph{Almgren's frequency} which is adapted to solutions of a sub-Laplacian on a Carnot group of arbitrary step $\bG$, which we call \emph{harmonic functions} hereafter. With this notion we have investigated some interesting new functionals associated with the frequency, and obtained  monotonicity formulas for the relevant harmonic functions, or for the solutions of a closely connected class of degenerate second order operators of Baouendi type, see \eqref{Ba} below. The results proved in this paper provide some new insight into the deep link existing between the growth properties of the frequency, and the local and global structure of the relevant harmonic functions in these non-elliptic, or \emph{subelliptic}, settings. 

To put our results in the proper historical perspective, we recall that in \cite{GL} E. Lanconelli and the first named author studied the problem of unique continuation on the Heisenberg group $\Hn$. Motivated by the cited works \cite{GLin1}, \cite{GLin2}, in \cite{GL} the authors introduced an analogue of  Almgren's frequency function which we now describe. Henceforth, given a function $u\in C^1(\Hn)$ we will indicate with 
$\uh$ the square of the length of its horizontal gradient, see \eqref{slhg} below. Denoting with $\rho(g) = (|z|^4 + 16 t^2)^{1/4}$ the anisotropic Koranyi gauge on $\Hn$, we let $B_r = \{g\in \Hn\mid \rho(g)<r\}$, and denote by $S_r = \p B_r$ its boundary. We define the \emph{Dirichlet integral} of a function $u$ in $B_r$  as follows
\begin{equation}\label{DH}
D(r) = \int_{B_r} \uh dg,
\end{equation}
where $\uh$ denotes the degenerate energy associated with the so-called horizontal gradient of $u$, see \eqref{slhg} below.
The \emph{height} of a function $u$ on $B_r$ is defined as
\begin{equation}\label{HH}
H(r) = \int_{S_r} u^2 \dsh,
\end{equation}
where we have denoted by $d\sigma_H$ the restriction to $S_r$ of the $H$-perimeter measure \`a la De Giorgi, see for instance \cite{DGN07}. The original definition of $H(r)$ in \cite{GL} was formulated in a different way, but we are casting it in the form \eqref{HH} since this is more in accordance with our general Definition \ref{D:DH} below. 

Given a harmonic function $u$, i.e., a solution of the equation  $\Delta_H u = 0$, where $\Delta_H$ indicates the Kohn sub-Laplacian on $\Hn$, the \emph{frequency} of $u$ was defined in \cite{GL} as
\begin{equation}\label{fH}
N(r) = N(u;r) = \frac{r D(r)}{H(r)}.
\end{equation}
Notice that, similarly to Almgren's frequency function, the function $N(u;r)$ is scale invariant, in the sense that
\begin{equation}\label{si}
N(\delta_r u;1) = N(u;r),
\end{equation}
where $\delta_r(z,t) = (rz,r^2 t)$ indicates the non-isotropic dilations in $\Hn$, see also the general Proposition \ref{P:si} below. In \cite{GL} the authors proved that the frequency function \eqref{fH} of a class of solutions to appropriate perturbations of the Kohn sub-Laplacian on $\Hn$ is locally bounded.

In this paper we propose a generalization of the functionals  \eqref{DH}, \eqref{HH} above for harmonic functions on any stratified nilpotent Lie group $\bG$ (also known as a \emph{Carnot group}). With such functionals we then form the relevant frequency function, which generalizes \eqref{fH}, and we study the connection between its properties and the local and global structure of the harmonic functions on $\bG$. In such general framework we needed an \emph{ad hoc} replacement for the the gauge balls in \eqref{DH}, \eqref{HH}. Motivated by the representation formulas in \cite{CGL}, we have chosen to work with the level sets of the fundamental solution of a sub-Laplacian $\Delta_H$ on $\bG$. Such fundamental solution was  constructed by Folland in \cite{F2}, see \eqref{balls}, \eqref{br} in Section \ref{S:prelim} and Definition \ref{D:DH}
 in Section \ref{S:DHN} below for the relevant definitions. 
 
 Remarkably, although the fundamental solution of a sub-Laplacian in a Carnot group is not, in general, explicitly known, we have nonetheless been able to derive, with the correct scalings, the basic first variation formulas for the relevant functionals entering in the definition of the frequency.  Proposition \ref{P:beauty} below is a first example. Such result allows to connect the Dirichlet integral $D(r)$ of a harmonic function $u$, to a surface integral which involves the infinitesimal generator of the non-isotropic group dilations, see \eqref{dilG}
 below. This immediately implies our Proposition \ref{P:hom} below stating that, similarly to what happens for a classical solution of $\Delta u = 0$ in $\Rn$, in a Carnot group $\bG$ the frequency of a harmonic function homogeneous of degree $\kappa$  is constant, and equal to $\kappa$. Here, the notion of homogeneity is tailored on the group dilations.

In Section \ref{S:fv} we connect the local boundedness of the frequency  to a deep property of harmonic functions, namely the \emph{doubling condition}, and the closely connected \emph{strong unique continuation property} (sucp). Unlike what happens   for classical harmonic functions in $\Rn$, harmonic functions on a Carnot group $\bG$ are not real-analytic in general, see the discussion at the end of Section \ref{S:fv}. Therefore, the analysis of the uniqueness properties of sub-Laplacians becomes an extremely delicate question which, regrettably, is still to present day largely not understood. An initial very interesting study of what can go wrong for smooth, even compactly supported, perturbations of sub-Laplacians was done by H. Bahouri in \cite{Ba}. However, Bahouri's work does not provide any evidence, in favor or to the contrary, about the fundamental, and largely open, question of whether \emph{any sub-Laplacian in a Carnot group possesses the sucp}. Our Theorem  \ref{T:Nb} below shows that a quantitative version of the sucp (the so-called \emph{doubling condition}) is intimately connected to the local boundedness of the frequency. We prove that these two properties are in fact equivalent. It is our hope to be able to address the general question of the local boundedness of the frequency in a future study.

In Section \ref{S:global} we prove that any harmonic function in a Carnot group which has globally bounded frequency $N(u,\cdot)$ must be a stratified polynomial of degree less than or equal the integral part of $||N(u,\cdot)||_{L^\infty(0,\infty)}$, see Theorem \ref{T:bf} below. This result was recently established for the above described frequency function introduced in \cite{GL} by H. Liu, L. Tian  X. Yang in the special setting of the Heisenberg group $\Hn$. Our result generalizes Theorem 1.1 in \cite{LTY} to all Carnot groups.  

In Section \ref{S:energy} we compute the first variation of the energy $D(r)$ of a harmonic function, see Proposition \ref{P:D'} below, and introduce the notion of \emph{discrepancy} of a function at a point, see Definition \ref{D:discrep} below. In Theorem \ref{T:mono} we show that if $u$ is a harmonic function with vanishing discrepancy, then its frequency is monotone nondecreasing. Furthermore, we prove in Proposition \ref{P:monhom} that if such a function has constant frequency equal to $\kappa$, then it must be homogeneous of degree $\kappa$. Combined with Proposition \ref{P:hom}, the latter result shows that, for a harmonic function $u$ having vanishing discrepancy, the frequency is constant and equal to $\kappa$ is and only if $u$ is homogeneous of degree $\kappa$. 

A fundamental open question is whether a harmonic function in a Carnot group possesses locally bounded frequency, or equivalently, whether harmonic functions have the strong unique continuation property. As we have shown in Theorem \ref{T:Nb} and Theorem \ref{T:sucp}, the boundedness of the frequency would suffice to establish the sucp. 
Although it is very tempting to conjecture that in a Carnot group $\bG$ the frequency of a harmonic function  is always locally bounded, presently we do not know whether this fundamental property is true.  In Section \ref{S:2theo} we prove two results that provide some interesting evidence in favor of this conjecture. The former, Theorem \ref{T:agnid}, states that in every group of Metivier type (and therefore, in particular, in any group of Heisenberg type) the frequency of a harmonic function is in fact locally bounded. For the Heisenberg group $\Hn$ the proof of this result was suggested to us by Agnid Banerjee. Here, we reproduce with his kind permission a generalization of his idea.
Theorem \ref{T:GLgen} provides an interesting sufficient condition, formulated on the discrepancy, for a harmonic function to have bounded frequency. It constitutes a generalization of a result which, in the special setting of the Heisenberg group $\Hn$, was proved in \cite{GL}.

In Section \ref{S:WM} we establish an interesting new monotonicity formula for the functional $\mathcal W_\kappa(u,r)$ defined in \eqref{W} below. The main thrust of this formula is that the derivative of $\mathcal W_\kappa(u,r)$ vanishes if and only if $u$ is homogenoeus of degree $\kappa$. Our main result is Theorem \ref{T:weissG} below which was inspired to a result that, for the classical Laplacian, was discovered  by Petrosyan and the first named author in \cite{GP}.
Using such result we establish Theorem \ref{T:N=k} below, which states that if a harmonic function $u$ in $\bG$ has vanishing discrepancy and constant frequency equal to $\kappa$, then $u$ must  be
a stratified solid harmonic of degree $\kappa$.

Section \ref{S:examples} is devoted to a further analysis of the discrepancy in the setting of groups of Heisenberg type. The characterization of the discrepancy provided by Lemma \ref{L:innerp} allows us to prove in Proposition \ref{P:cvd} that in a group of Heisenberg type a harmonic function has vanishing discrepancy if and only if in the exponential coordinates it is a solution of the Baouendi operator $\Ba$ in \eqref{Ba} below with $\alpha = 1$.

In Section \ref{S:bg} we finally turn to monotonicity properties of the \emph{Baouendi type operators} \begin{equation}\label{Ba}
\Ba u = \Delta_z u + \frac{|z|^{2\alpha}}{4} \Delta_t u,\ \ \ \ \ \ \alpha>0.
\end{equation} 
Here, for given $m, k\in \mathbb N$, we have let $z\in \R^m$ and $t\in \R^k$, and henceforth we will indicate $\R^N = \R^m \times \R^k$. The operator $\Ba$ is uniformly elliptic away from the $k$-dimensional manifold $\mathcal M = \{0\}\times \R^k \subset \R^N$, but its ellipticity degenerates as one approaches $\mathcal M$. When $\alpha = 2 \ell$, with $\ell\in \mathbb N$, then $\Ba$ is an operator of H\"ormander type, and therefore by the results in \cite{H} it is hypoelliptic. The operators \eqref{Ba} belong to a class of operators first studied by S. Baouendi in \cite{B}, and later by V. Grushin in \cite{Gr1} and \cite{Gr2}. A deep study of the local properties of solutions of equations modeled on \eqref{Ba} was conducted in
the pioneering work of Franchi and Lanconelli \cite{FL}, see also the subsequent work of Franchi and Serapioni \cite{FS}, and the references therein. 

The case $\alpha = 1$ in \eqref{Ba} has a special significance since in such case the resulting operator is connected with the real part of the Kohn-Spencer sub-Laplacian $\Delta_H$ on the Heisenberg group $\Hn$. In the real coordinates of $\R^{2n+1}$ the latter operator is given by
\begin{equation}\label{slHn}
\Delta_H = \Delta_z + \frac{|z|^2}{4} \p_{tt} + \p_t \sum_{j=1}^n \left(x_j \p_{y_j} - y_j \p_{x_j}\right).
\end{equation}
If we let $\Theta = \sum_{j=1}^n \left(x_j \p_{y_j} - y_j \p_{x_j}\right)$, then it is clear that a harmonic function in $\Hn$, i.e., a solution of $\Delta_H u = 0$, for which $\Theta u = 0$, is also a solution of $\mathcal B_1 u = 0$. We note that a function $u$ in $\Hn$ has $\Theta u = 0$ if and only if $u(e^{i\vt}z,t) = u(z,t)$ for every $\vt\in \R$, and for every $(z,t) \in \mathbb C^{n}\times \R$. When in particular $n=1$, then $\Theta u = 0$ if and only if $u(z,t) = u^\star(|z|,t)$.

We recall that for solutions of the operators \eqref{Ba} a generalization of Almgren's monotonicity formula was proved in \cite{G}. In Section \ref{S:bg} we use the results in \cite{G} to establish some new monotonicity results. The former, Theorem \ref{T:weissB} below, is the counterpart of Theorem \ref{T:weissG}, except that, interestingly, for the solutions of the Baouendi operators $\Ba$ we do not have the additional hypothesis of vanishing discrepancy. Our second main result is Theorem \ref{T:MBa} below, which is inspired to a monotonicity formula for the classical Laplacian in \cite{GP}.

\medskip

\noindent \textbf{Acknowledgement:} The work of Ermanno Lanconelli has been a constant source of inspiration. One of us, N.G., has not only benefited from such inspiration, but also shared with Ermanno a long friendship, countless hours of mathematical collaborations and discussions. On the occasion of his birthday, we dedicate this paper to him with affection. 

We also would like to thank Donatella Danielli for several helpful conversations and suggestions, and Agnid Banerjee for having pointed to our attention the second part of Theorem \ref{T:Nb}. The idea of the proof of Theorem \ref{T:agnid} below is also his, and we present it in this paper with his kind permission.

\section{Preliminaries}\label{S:prelim}

We recall that a Carnot group of step $r$ is a connected, simply connected Lie
group $\bG$ whose Lie algebra $\bg$ admits a stratification $\bg=
V_1 \oplus \cdots \oplus V_r$ which is $r$-nilpotent, i.e.,
$[V_1,V_j] = V_{j+1},$ $j = 1,...,r-1$, $[V_j,V_r] = \{0\}$, $j =
1,..., r$. We assume henceforth that $\mathfrak g$ is endowed with a
scalar product $<\cdot,\cdot>_\bg$ with respect to which the $V_j's$
are mutually orthogonal. A trivial example of (an Abelian) Carnot
group is $\bG = \Rn$, whose Lie algebra admits the trivial
stratification $\bg = V_1 = \Rn$. The simplest non-Abelian example
of a Carnot group of step $r=2$ is the above mentioned $(2n+1)$-dimensional
Heisenberg group $\Hn$.
Given a Carnot group $\bG$, by the above assumptions on the Lie
algebra one immediately sees that any basis of the \emph{horizontal
layer} $V_1$ generates the whole $\bg$. We will respectively denote
by
\begin{equation}\label{LT}
L_g(g')\ =\ g \ g'\ ,\quad\quad\quad\quad R_g(g')\ =\ g' \ g\ ,
\end{equation}
the operators of left- and right-translation by an element $g\in
\bG$.

The exponential mapping $\exp : \bg \to \bG$ defines an analytic
diffeomorphism onto $\bG$. We recall the 
Baker-Campbell-Hausdorff formula, see, e.g., sec. 2.15 in \cite{V},
\begin{equation}\label{BCH}
\exp(\xi)  \exp(\eta) = \exp{\bigg(\xi + \eta + \frac{1}{2}
[\xi,\eta] + \frac{1}{12} \big\{[\xi,[\xi,\eta]] -
[\eta,[\xi,\eta]]\big\} + ...\bigg)}\ ,
\end{equation}
where the dots indicate commutators of order four and higher. Each
element of the layer $V_j$ is assigned the formal degree $j$.
Accordingly, one defines dilations on $\bg$ by the rule
\[
\Delta_\lambda \xi = \lambda \xi_1 + ... + \lambda^r
\xi_r,
\]
provided that $\xi = \xi_1 + ... + \xi_r \in \bg$, with $\xi_j\in
V_j$. Using the exponential mapping $\exp : \bg \to \bG$, these
anisotropic dilations are then tansferred to the group $\bG$ as
follows
\begin{equation}\label{dilG}
\delta_\lambda(g) = \exp \circ \Delta_\lambda \circ
\exp^{-1} g.
\end{equation}
Throughout the paper we will indicate by $dg$ the bi-invariant
Haar measure on $\bG$ obtained by lifting via the exponential map
$\exp$ the Lebesgue measure on $\bg$. We let $m_j =$ dim$ V_j$, $j=
1,...,r$, and denote by $N = m_1 + ... + m_r$ the topological
dimension of $\bG$. For ease of notation, we agree from now on to indicate with $m$, instead of $m_1$, the dimension of the horizontal layer $V_1$ of $\bg$. One easily checks that
\begin{equation}\label{cim}
(d\circ\delta_\lambda)(g) = \lambda^Q dg ,
\quad\quad\text{where}\quad Q = \sum_{j=1}^r j m_j.
\end{equation}
The number $Q$, called the \emph{homogeneous dimension} of $\bG$,
plays an important role in the analysis of Carnot groups. In the
non-Abelian case $r>1$, one clearly has $Q>N$.

Let $\{e_1,...,e_m\}$ indicate an orthonormal basis of the first layer $V_1$ of the Lie algebra, and define the left-invariant $C^\infty$ vector fields on $\bG$ by the formula
\[
X_i(g) = dL_g(e_i),\ \ \ \ i=1,...,m,
\]
where $dL_g$ indicates the differential of $L_g$. We assume throughout this paper that $\bG$ is endowed with a left-invariant Riemannian metric with respect to which the vector fields $\{X_1,...,X_m\}$ are orthonormal. The sub-Laplacian associated with the basis $\{e_1,...,e_m\}$ is defined by the formula
\begin{equation}\label{slG}
\Delta_H u = \sum_{i=1}^m X_i^2 u.
\end{equation}
A distribution $u$ is called harmonic in $\bG$ if $\deh u = 0$ in $\mathcal D'(\bG)$. Since the vector fields $X_1,...,X_m$ and their commutators up to step $r$ generate the whole Lie algebra of left-invariant vector fields on $\bG$, thanks to H\"ormander's theorem, see \cite{H}, $\deh$ is hypoelliptic (but not, in general, real analytic hypoelliptic). Thus, a harmonic distribution $u$ can be modified on a set of measure zero so that it coincides with a $C^\infty$ harmonic function. Henceforth, we will indicate with $e\in \bG$ the group identity.

Let now $\Gamma(g,g') = \Gamma(g',g)$ be a positive
fundamental solution of $-\Delta_H$. Such distribution is left-translation invariant, i.e., we can write
\[
\Gamma(g,g') = \tilde \Gamma(g^{-1}\circ g').
\]
For every $r>0$, let 
\begin{equation}\label{balls}
B_r = \left\{g\in \bG \mid \Gamma(g,e)>\frac{1}{r^{Q-2}}\right\}.
\end{equation}
It was proved by Folland in \cite{F2} that the distribution $\tilde \Gamma(g)$ is homogeneous of degree $2-Q$ with respect to the non-isotropic dilations \eqref{dilG}. This implies that, if we define
\begin{equation}\label{rho}
\rho(g) = \Gamma(g)^{- 1/(Q-2)},
\end{equation}
then the function $\rho$ is homogeneous of degree one. We obviously have from \eqref{balls}
\begin{equation}\label{br}
B_r = \{g\in \bG\mid \rho(g)<r\}.
\end{equation}
Henceforth, we will use the notation $S_r = \p B_r$.

The position
\begin{equation}\label{pd}
d(g,g') = \rho(g^{-1}\circ g'),
\end{equation}
defines a pseudo-distance on $\bG$. Using such pseudo-distance, we set
\[
B_r(g) = \{g'\in \bG\mid d(g',g)<r\},\ \ \ \ S_r(g) = \p B_r(g).
\]
We observe that if $|\cdot|$ denotes a non-isotropic gauge on $\bG$, then there exist $0<\beta<\alpha<\infty$, depending only on $\bG$, such that
\begin{equation}\label{gaugecomp}
\alpha^{-1} |g| \le \rho(g) \le \beta^{-1} |g|,\ \ \ \ \ g\in \bG.
\end{equation}
As a consequence, with obvious meaning of the notations, we have
\begin{equation}\label{alphabeta}
B^{|\cdot|}_{\beta r} \subset B_r \subset B^{|\cdot|}_{\alpha r}.
\end{equation}

A function $u$ on a Carnot group $\bG$ is called \emph{homogeneous of degree} $\kappa$ with respect to the non-isotropic group dilations  $\{\delta_\la\}_{\la>0}$ in \eqref{dilG} if
\begin{equation}\label{homG}
u(\delta_\lambda g) = \lambda^\kappa u(g),\ \ \ \ \ g\in \bG,\ \la>0.
\end{equation}   
Throughout this paper we will indicate with $Z$ the infinitesimal generator of the non-isotropic dilations \eqref{dilG}. Notice that such $C^\infty$ vector field is characterized by the property 
\begin{equation}\label{dZ}
 \frac{d}{dr} u(\delta_r g) = \frac{1}{r} Zu(\delta_r g).
 \end{equation}

One can show that a function $u\in C^1(\bG)$ is homogeneous of degree $\kappa$ if and only if the following Euler formula holds
\begin{equation}\label{E0}
Zu = \kappa u.
\end{equation}

We close this section by noting the following fact about the function $\rho$ defined by \eqref{rho}.

\begin{prop}\label{P:radial}
Let $f:(0,\infty)\to \R$ be a $C^2$ function, and define $u(g) = f(\rho(g))$. Then, one has
\[
\Delta_H u = |\nh \rho|^2 \left\{f''(\rho) + \frac{Q-1}{\rho} f'(\rho)\right\},\ \ \ \ \ \text{in}\ \bG\setminus \{e\}.
\]
\end{prop}

\begin{proof}
It is an elementary computation based on \eqref{rho} and the chain rule. The details are left to the reader.

\end{proof}

\section{Energy, height and frequency of harmonic functions}\label{S:DHN}

In this section we introduce the Dirichlet integral, the height and the frequency of a function on a Carnot group $\bG$, and establish some first basic properties of these functionals. We begin with the relevant definition. In what follows, for a given function $u$ on $\bG$ we indicate with $\nh u$ the horizontal gradient of $u$ given by
\[
\nh u = \sum_{j=1} X_j u\ X_j.
\]
We also let
\begin{equation}\label{slhg}
|\nabla_H u|^2 = \sum_{j=1}^m (X_ju)^2.
\end{equation}

\begin{dfn}\label{D:DH}
Given a function $u$ in a ball $B_R(g_0)\subset \bG$, for every $0<r<R$ we define its \emph{Dirichlet integral} 
\begin{equation}\label{D}
D_{g_0}(u,r) = \int_{B_r(g_0)} |\nh u|^2 dg.
\end{equation}
The \emph{height} function of $u$, see \eqref{HH} above, is defined as
\begin{equation}\label{H}
H_{g_0}(r) = \int_{S_r(g_0)} u^2 |\nabla_Hd(\cdot,g_0)| d\sigma_H,
\end{equation}
where we have denoted by $d\sigma_H$ the $H$-perimeter measure on $S_r(g_0)$, and with $d(g,g_0)$ the pseudo-distance \eqref{pd}. The \emph{frequency} of $u$ (with respect to $g_0\in \bG$) is defined as 
\begin{equation}\label{N}
N_{g_0}(u;r) = \frac{r D_{g_0}(u,r)}{H_{g_0}(u,r)},\ \ \ \ \ 0<r<R.
\end{equation}
When $g_0 = e$, the group identity, then we simply write $D(u,r), H(u,r)$ and $N(u,r)$. In such case, the kernel $|\nabla_Hd(\cdot,g_0)|$ will be simply indicated with $|\nh \rho|$.
\end{dfn}

By the left-translation invariance of the Haar measure on $\bG$ (see \cite{CG}) and of the $H$-perimeter measure (see \cite{forum}) we immediately see that the frequency is invariant with respect to left-translations, i.e.,
\begin{equation}\label{Ninv}
N_{g_0}(u,r) = N(u\circ L_{g_0},r).
\end{equation}
 
\begin{rmrk}\label{R:u}
In view of \eqref{Ninv} we can focus our analysis of the frequency on balls which are centered at the group identity. Also, when the function $u$ is fixed throughout the discussion, we will simply write $D(r), H(r)$ and $N(r)$, instead of $D(u,r), H(u,r)$ and $N(u,r)$.
\end{rmrk}

A first important property of the frequency is represented by the following scale invariance.

\begin{prop}\label{P:si}
Let $\Delta_H u = 0$ in $B_R$, then for every $0<\la <R/r$ one has
\[
N(u\circ \delta_{\la},r) = N(u,\la r).
\]
\end{prop}

\begin{proof}
Hereafter, if $N$ indicates the topological dimension of the group $\bG$ we will denote by $H_{N-1}$ the standard $(N-1)-$dimensional Hausdorff measure on $\bG$. We begin by recalling, see \cite{DGN07}, that the $H$-perimeter measure on the boundary of the $C^\infty$ 
domain $B_r$ is given by
\begin{equation}\label{sigmaH}
d\sigma_H = |\bN_H| dH_{N-1},
\end{equation}
where, with $\nu$ being the (Riemannian) outer unit normal on $B_r$, one has
\begin{equation}\label{Nhv}
\bN_H = \sum_{i=1}^m <X_i,\nu>X_i = \frac{1}{|\nabla \rho|}\sum_{i=1}^m X_i\rho X_i =  \frac{1}{|\nabla \rho|} \nh \rho.
\end{equation}
We thus obtain on $S_r$
\begin{equation}\label{Nh}
|\bN_H| = \frac{\rhh}{|\nabla \rho|}.
\end{equation}
Recall that the functions $g\to \rho(g)$ and $g\to |\nh \rho(g)|$ are respectively homogeneous of degree one and zero with respect to \eqref{dilG}. We also recall, see for instance \cite{forum}, that the $H$-perimeter measure scales correctly with respect
to the nonisotropic dilations \eqref{dilG}, in the sense that
\begin{equation}\label{Hper}
d\sigma_H(\delta_\lambda(g)) = \lambda^{Q-1} d\sigma_H(g).
\end{equation}
Using the properties \eqref{cim} and \eqref{sigmaH}-\eqref{Hper}, we now have
\[
r \int_{B_r} |\nh (u\circ \delta_\la)|^2 dg = (r\la) \la^{1-Q} \int_{B_{r\la}} |\nh u|^2 dg',
\]
 \[
 \int_{S_r} (u\circ \delta_\la)^2 |\nh \rho| d\sigma_H = \la^{1-Q} \int_{S_{r\la}} u^2  |\nh \rho| d\sigma_H,
 \]
 and thus the desired conclusion immediately follows.

\end{proof}

Our second result generalizes a corresponding property of the Dirichlet integral of a classical harmonic function in $\Rn$.

\begin{prop}\label{P:beauty}
Let $\bG$ be a Carnot group of arbitrary step and let $Z$ be the infinitesimal generator of the group dilations \eqref{dilG} above. If $u$ is harmonic function in $B_R$, then for
every $0<r<R$ we have
\[
D(r)  = \int_{S_r} u \frac{Zu}{r}
 |\nabla_H\rho| d\sigma_H.
\]
\end{prop}

\begin{proof}

We start with an even more general situation, and consider
H\"ormander type vector fields $\{X_1,...,X_m\}$ in $\R^N$. The
following formula is valid for $\psi \in C^{\infty}(\R^N)$ and
$0<t\leq R$, see \cite{CGL},
\begin{equation}\label{psi}
\psi(x) = \int_{\partial \Om(x,t)} \psi(y) \frac{|\nabla_H
\Gamma(x,y)|^2}{|\nabla \Gamma(x,y)|} dH_{N-1}(y) -
\int_{\Om(x,t)} \Delta_H \psi(y)
\big[\Gamma(x,y)-\frac{1}{t}\big] dy,
\end{equation}
where $\Delta_H = - \sum_{i=1}^m X^*_iX_i$, $\Gamma(x,y)$ is a positive
fundamental solution of $-\Delta_H$, $\Om(x,t) = \{y\in \R^N\mid
\Gamma(x,y)>\frac{1}{t}\}$, and $H_{N-1}$ denotes the standard $(N-1)$-dimensional Hausdorff measure in $\R^N$. 

Now, suppose we are in a Carnot group
$\bG$. Note that, since in a Carnot group $X_i^* = - X_i$, the operator $\Delta_H$ coincides with \eqref{sl} above. Let $B_r = \{g\in \bG \mid \Gamma(g,e)>\frac{1}{r^{Q-2}}\}$,
i.e. $B_r = \Om(e,r^{Q-2})$. Setting $t=r^{Q-2}$ in  \eqref{psi} we
obtain
\begin{equation}\label{psiG}
\psi(e)\ =\ \int_{S_r}\ \psi(g)\
\frac{|\nabla_H\Gamma(g)|^2}{|\nabla\Gamma(g)|}\ dH_{N-1}(g)\ -\
\int_{B_r}\ \Delta_H \psi(g)\ \big[\Gamma(g)-\frac{1}{r^{Q-2}}\big]\
dg.
\end{equation}

Introduce now the function $\rho(g) = \Gamma(g)^{-1/(Q-2)}$ in \eqref{rho} above.
With $E(s) = s^{Q-2}$, we easily recognize that
\[
\nabla_H \Gamma = - \frac{E'(\rho)}{E(\rho)^2} \nabla_H\rho .
\]

This gives
\[
\frac{|\nabla_H \Gamma|^2}{|\nabla \Gamma|} =
 \frac{E'(\rho)}{E(\rho)^2}  \frac{|\nabla_H\rho|^2}{|\nabla \rho|} = \frac{Q-2}{\rho^{Q-1}}  \frac{|\nabla_H\rho|^2}{|\nabla \rho|} .
\]

Substitution in \eqref{psiG} gives
\begin{equation}\label{2psiG}
\psi(e) = \frac{Q-2}{r^{Q-1}} \int_{S_r} \psi(g)
\frac{|\nabla_H\rho(g)|^2}{|\nabla\rho(g)|} dH_{N-1}(g) -
\int_{B_r} \Delta_H \psi(g) \big[\frac{1}{\rho^{Q-2}}
-\frac{1}{r^{Q-2}}\big] dg.
\end{equation}

Suppose now that $\psi = u^2$, with $\Delta_H u = 0$ in $\bG$, then
\eqref{2psiG} gives
\begin{equation}\label{usqG}
u(e)^2 = \frac{Q-2}{r^{Q-1}} \int_{S_r}\ u(g)^2
\frac{|\nabla_H\rho(g)|^2}{|\nabla\rho(g)|} dH_{N-1}(g) - 2
\int_{B_r} |\nabla_H u(g)|^2 \big[\frac{1}{\rho^{Q-2}}
-\frac{1}{r^{Q-2}}\big] dg.
\end{equation}

Notice that since $\Gamma$ is homogeneous of degree $2-Q$, the
function $\rho$ in \eqref{rho} is homogeneous of degree one, and therefore
$|\nabla_H \rho|$ (or any of its powers) is homogeneous of degree
zero. Differentiating \eqref{usqG} with respect to $r$ we obtain
\begin{equation}\label{2usqG}
\frac{d}{dr} \frac{1}{r^{Q-1}} \int_{S_r} u(g)^2
\frac{|\nabla_H\rho(g)|^2}{|\nabla\rho(g)|} dH_{N-1}(g) =
\frac{2}{r^{Q-1}} \int_{B_r} |\nabla_H u(g)|^2 dg.
\end{equation}

Rewriting \eqref{2usqG} we obtain
\begin{equation}\label{3usqG}
D(r) = \int_{B_r} |\nabla_H u(g)|^2 dg = \frac{r^{Q-1}}{2}
\frac{d}{dr} \frac{1}{r^{Q-1}} \int_{S_r} u(g)^2
\frac{|\nabla_H\rho(g)|^2}{|\nabla\rho(g)|} dH_{N-1}(g).
\end{equation}

Using the properties  \eqref{sigmaH}-\eqref{Hper} above, after a rescaling we obtain
\begin{align}\label{kf}
& \frac{1}{r^{Q-1}} \int_{S_r}\ u(g)^2
\frac{|\nabla_H\rho(g)|^2}{|\nabla\rho(g)|} dH_{N-1}(g) =
\frac{1}{r^{Q-1}} \int_{S_r} u(g)^2\ |\nabla_H\rho(g)|
d\sigma_H(g) \\
& =  \int_{S_1} u(\delta_r(g))^2 |\nabla_H\rho(g)|
d\sigma_H(g). \notag
\end{align}

Now we use \eqref{kf}, and \eqref{dZ} to obtain from \eqref{3usqG}
\begin{align}\label{ff}
D(r) = & \frac{r^{Q-1}}{2} \frac{d}{dr} \int_{S_1}
u(\delta_r(g))^2 |\nabla_H\rho(g)| d\sigma_H(g)
\\
& = \frac{r^{Q-1}}{r} \int_{S_1} u(\delta_r(g))
Zu(\delta_r(g)) |\nabla_H\rho(g)| d\sigma_H(g), \notag\\
& = \frac{1}{r} \int_{S_r} u(g) Zu(g) |\nabla_H\rho(g)|
d\sigma_H(g), \notag
\end{align} where we have denoted by $Z$ the
generator of the nonisotropic dilations. This completes the proof.

\end{proof}

We note for future use the following basic consequence of \eqref{2usqG} above.

\begin{cor}\label{C:averages}
Let $\Delta_H u = 0$ in $B_R\subset \bG$. Then, the averages
\[
r\ \longrightarrow\  \frac{1}{r^{Q-1}} \int_{S_r} u^2
\frac{|\nabla_H\rho|^2}{|\nabla\rho|} dH_{N-1},\ \ \ \ 0<r<R,
\]
are nondecreasing. This implies, in particular, that
\begin{equation}\label{solsph}
\int_{B_{r}} u^2 |\nabla_H \rho|^2 dg \le \frac{r}{Q} \int_{S_r} u^2
\frac{|\nabla_H\rho|^2}{|\nabla\rho|} dH_{N-1}
\end{equation}
\end{cor}

\begin{proof}
By the first part of Corollary \ref{C:averages}, and the coarea formula, we have
\begin{align*}
\int_{B_{r}} u^2 |\nabla_H \rho|^2 dg' & = \int_0^r t^{Q-1} \frac{1}{t^{Q-1}} \int_{S_t} u^2
\frac{|\nabla_H\rho|^2}{|\nabla\rho|} dH_{N-1} dt
\\
& \le \frac{r}{Q} \int_{S_r} u^2
\frac{|\nabla_H\rho|^2}{|\nabla\rho|} dH_{N-1}.
\end{align*}

\end{proof}

\begin{lemma}\label{L:nondeg}
Let $u$ be a solution of $\Delta_H u = 0$ in $B_R\subset \bG$. For any $0<r<R$, either $u\equiv 0$ in $B_r$, or $H(r)\not = 0$.
\end{lemma}

\begin{proof}
Suppose $H(r) = 0$ for some $r\in (0,R)$. Then, $u\equiv 0$ on $S_r$. By Proposition \ref{P:beauty} we must have $D(r) = 0$, and therefore $\nabla_H u \equiv 0$ in $B_r$. Since $X_1,...,X_m$ generate the Lie algebra of all left-invariant vector fields, we conclude  that $u\equiv 0$  in $B_r$.

\end{proof}

\begin{rmrk}\label{R:nonvan}
Henceforth, when we speak of the frequency of a harmonic function $u$ in $B_R$ we will always tacitly assume that for no $0<r<R$ the function $u$ vanishes identically in $B_r$. In view of  Lemma \ref{L:nondeg} this assumptions guarantees that $H(r)\not= 0$ for every $0<r<R$, and therefore the frequency $N(r)$ of $u$ is well defined in $(0,R)$.
\end{rmrk}

With Proposition \ref{P:beauty} in hands, we immediately obtain the following generalization of a result which, in the special case of the Heisenberg group $\Hn$, was proved in \cite{GL}.

\begin{prop}\label{P:hom}
Let $\bG$ be a Carnot group of arbitrary step, and suppose that $u$ be harmonic in $B_R\subset \bG$. If $u$ is homogeneous of degree $\kappa\ge 0$, then 
$N(u;r) = \kappa,$ for every $r\in (0,R)$.
\end{prop}

\begin{proof}
Suppose that $u$ be harmonic in $B_R$ and homogeneous of degree $\kappa$. Then, we have by  Proposition \ref{P:beauty} and \eqref{E0} 
\[
D(r) = \int_{\p B_r} u \frac{Zu}{r}
 |\nabla_H\rho| d\sigma_H = \frac{\kappa}{r} \int_{\p B_r} u^2
 |\nabla_H\rho| d\sigma_H = \frac{\kappa}{r} H(r).
\]
From \eqref{N} we conclude
\[
N(r) = \kappa,\ \ \ \ \ 0<r<R.
\]

\end{proof}

In Proposition \ref{P:monhom} below we will establish a partial converse to Proposition \ref{P:hom}.


\section{Local consequences of the boundedness of the frequency}\label{S:fv}

In this section we study the local  behavior of harmonic functions having bounded frequency. 
We begin with the following simple, but very important consequence of formula \eqref{2usqG} above.

\begin{lemma}\label{L:H'}
Suppose that $u$ be harmonic in $B_R$. Then, for every $r\in (0,R)$ one has
\[
H'(r) = \frac{Q-1}{r} H(r) + 2 D(r).
\]
\end{lemma}

\begin{proof}
Formula \eqref{2usqG} can be rewritten as follows
\[
\frac{d}{dr} \frac{1}{r^{Q-1}} H(r) =
\frac{2}{r^{Q-1}} D(r).
\]
Differentiating the left-hand side, we obtain
\[
- \frac{Q-1}{r} H(r) + H'(r) = 2 D(r),
\]
which gives the desired conclusion.

\end{proof}

For a function $u\in C(\bG)$ and $r>0$ we now define the following mean value operators
\begin{equation}\label{smvf0}
M_r u(g) = \frac{Q-2}{Q} {r^{-Q}} \int_{B_r(g)} u(g')
|\nabla_H d(g',g)|^2 dg'.
\end{equation}
Notice that if $\Delta_H u = 0$, then we obtain from \eqref{2psiG} above
\begin{equation}\label{mvf}
u(g) = \frac{Q-2}{t^{Q-1}} \int_{\partial B_t(g)} u(g')
\frac{|\nabla_H d(g' \circ g)|^2}{|\nabla d(g' \circ g)|} dH_{N-1}(g').
\end{equation}
Integrating \eqref{mvf} over the interval $(0,r)$, and using the co-area formula, we find
\begin{equation}\label{smvf}
u(g) = \frac{Q-2}{Q} {r^{-Q}} \int_{B_r(g)} u(g')
|\nabla_H d(g' \circ g)|^2 dg' = M_r u(g).
\end{equation}

Before proving Theorem \ref{T:Nb} we recall the following key Lemma 3.1 from \cite{BL} which allows to connect the averages \eqref{smvf} on balls centered at different points.

\begin{lemma}\label{L:trick}
Let $\bG$ be a Carnot group and let $f\ge 0$ be a continuous function on $\bG$. There exist constants $C, \lambda, \Lambda >0$, depending only on $\bG$, such that for every $r>0$ one can find $g_0\in \bG$ such that
\begin{itemize}
\item[(i)] $g_0\in \p B_{\lambda r}$;
\item[(ii)] $M_r f(g) \le C M_{\Lambda r} f(g_0)$,\ \ \ for every\ $g\in B_r$;
\item[(iii)]  $M_r f(g_0) \le C M_{\Lambda r} f(g)$,\ \ \ for every\ $g\in B_r$.
\end{itemize}
\end{lemma}

The next result establishes a remarkable equivalence between the local boundedness of the frequency and the order of vanishing at a point of a harmonic function. We thank Agnid Banerjee for pointing the second part to our attention.

\begin{thrm}\label{T:Nb}
Let $u$ be a harmonic function in $B_R$ and suppose that $N(u,\cdot) \in L^\infty(0,R)$, with $\kappa = ||N(u,\cdot)||_{L^\infty(0,R)}$. Then, there exists a universal $C^\star >0$ such that
\begin{equation}\label{DC}
\int_{B_{2r}} u^2 dg \le C^\star 2^{Q+2\kappa}  \int_{B_r} u^2 dg, \ \ \ \ \ 0<r<R/2.
\end{equation}
Vice-versa, suppose there exist a constant $\overline C>0$ such that
\begin{equation}\label{DC2}
\int_{B_{2r}} u^2 dg \le \overline C \int_{B_r} u^2 dg, \ \ \ \ \ 0<r<R.
\end{equation}
Then, $N(u,\cdot) \in L^\infty(0,R/2)$, and there is a universal $C'''>0$ such that $||N(u,\cdot)||_{L^\infty(0,R/2)} \le \overline C C'''$.
\end{thrm}

\begin{proof}
We begin by rewriting the conclusion of Lemma \ref{L:H'} in the following way
\begin{equation}\label{H'2}
\frac{d}{dr} \log \frac{H(t)}{t^{Q-1}} = 2 \frac{N(t)}{t},\ \ \ \ \ 0<t<R.
\end{equation}
Integrating \eqref{H'2} between $r$ and $2r$, with $0<r<R/2$, we find
\[
\log \frac{H(2r)}{2^{Q-1} H(r)} = 2 \int_r^{2r} N(t) \frac{dt}{t} \le 2 \kappa \log 2. 
\]
Exponentiating, we obtain
\begin{equation}\label{dH}
H(2t) \le H(t)\ 2^{Q -1+ 2\kappa},\ \ \ \ \ 0<t<R/2.
\end{equation}
We next fix $r\in (0,R/2)$, integrate the latter inequality over the interval $(0,r)$, and apply Federer's coarea formula
to find
\begin{equation}\label{DD}
\int_{B_{2r}} u^2 |\nh \rho|^2 dg' \le 2^{Q+2\kappa}  \int_{B_r} u^2 |\nh \rho|^2 dg'.
\end{equation}
Let now $0<r<R/2$. There exists and a universal $C'>0$ such that
\begin{align*}
\int_{B_{2r}} u^2 dg' & \le C' r^{Q}  \underset{B_{2r}}{\sup}\ u^2
\end{align*}
At this point we use Lemma \ref{L:trick} above. Let $g\in \overline B_{2r}$ be such that
\[
\underset{B_{2r}}{\sup}\ u^2 = u(g)^2.
\]
Since $u$ is harmonic, for every $\alpha>0$ we have by \eqref{smvf} 
\[
|u(g)| \le M_{\alpha r}(|u|)(g).
\]
By (ii) in Lemma \ref{L:trick} we can find $g_0\in \p B_{\lambda \alpha r}$ such that
\[
M_{\alpha r}(|u|)(g) \le C M_{\Lambda \alpha r}(|u|)(g_0).
\]
By (iii) in Lemma \ref{L:trick} we have
\[
M_{\Lambda \alpha r}(|u|)(g_0) \le C M_{\Lambda^2 \alpha r}(|u|)(e).
\]
From this chain of inequalities and the Cauchy-Schwarz inequality we conclude for some universal $C>0$
\begin{equation}\label{lact}
r^{-Q} \int_{B_{2r}} u^2  dg'  \le \underset{B_{2r}}{\sup}\ u^2 =  u(g)^2 \le C  r^{-Q} \int_{B_{\Lambda^2 \alpha r}} u^2 |\nabla_H \rho|^2 dg',\ \ \ \ 0<r<R/(\Lambda^2 \alpha).
\end{equation}
If we now choose $\alpha>0$ such that $\alpha \Lambda^2 = 2$, we obtain from \eqref{lact} and \eqref{DD}.
\[
\int_{B_{2r}} u^2  dg'  \le C 2^{Q+2\kappa}  \int_{B_r} u^2 |\nh \rho|^2 dg',\ \ \ \ \ 0<r<R/2.
\]
To reach the desired conclusion \eqref{DC}, all is left at this point is to observe that, since $|\nh \rho|^2$ is homogeneous of degree zero in $\bG$, there exists a universal constant $C'>0$ such that
\[
|\nh \rho(g)|^2 \le C',\ \ \ \ \ \text{for every}\ g\in \bG\setminus \{e\}.
\] 

We now establish the opposite implication in the proposition. Suppose that \eqref{DC2} hold. The subelliptic Caccioppoli inequality, see \cite{CGL}, and the hypothesis \eqref{DC2},
 give for a universal $C'>0$ and every $0<r<R/2$,
\begin{equation}\label{lact2}
D(r) \le \frac{C'}{r^2} \int_{B_{2r}} u^2 dg' \le \frac{\overline C C'}{r^2} \int_{B_{r}} u^2 dg'.
\end{equation}
We now use \eqref{lact} (with $\alpha \Lambda^2 = 2$ and $2r$ replaced by $r$), to reach the conclusion that
\[
D(r) \le \frac{\overline C C''}{r^2} \int_{B_{r}} u^2 |\nabla_H \rho|^2 dg'.
\]
At this point, we use the harmonicity on $u$ one more time and observe that, thanks to \eqref{solsph} in  Corollary \ref{C:averages}, we have
\begin{align*}
\int_{B_{r}} u^2 |\nabla_H \rho|^2 dg' \le \frac{r}{Q} H(r).
\end{align*}
Substituting into \eqref{lact2} we finally obtain for $0<r<R/2$
\[
N(r) \le \overline C C''',
\]
for a universal $C'''>0$. This proves $||N(u,\cdot)||_{L^\infty(0,R/2)} \le \overline C C'''$, thus completing the proof.

\end{proof}

As it is by now well-known, see \cite{GLin1}, \cite{GLin2}, the \emph{doubling condition} \eqref{DD} has deep implications on the order of vanishing of a harmonic function at one point. Let us introduce the relevant definition.

\begin{dfn}\label{D:infty}
We say that a harmonic function $u$ in $B_R(g_0)$ vanishes to infinite order at $g_0\in \bG$ if for every $p\in \mathbb N$ one has
\[
\int_{B_r(g_0)} u^2 dg = O(r^p),\ \ \ \ \text{as}\ r\to 0^+.
\]
\end{dfn}

We have the following result.

\begin{thrm}\label{T:sucp}
Let $u$ be a harmonic function in $B_R(g_0)$, and suppose that $N_{g_0}(u,\cdot) \in L^\infty(0,R)$. If $u$ vanishes of infinite order at $g_0$, then $u\equiv 0$ in $B_{R}(g_0)$.
\end{thrm}

\begin{proof}
By left-translation, see \eqref{Ninv} above, we can assume without loss of generality that $g_0 = e\in \bG$. Let  $\kappa = ||N_{g_0}(u,\cdot)||_{L^\infty(0,R)}$, and denote by $K = C^\star 2^{Q+2\kappa}$ the constant in \eqref{DC} in Theorem \ref{T:Nb}. Fix $R_0<R$. We thus have 
\begin{align*}
\int_{B_{R_0}} u^2 dg & \le K  \int_{B_{2^{-1}R_0}} u^2 dg \le ... \le K^\ell \int_{B_{2^{-\ell}R_0}} u^2 dg
\\
& = K^\ell |B_{2^{-\ell} R_0}|^\gamma |B_{2^{-\ell}R_0}|^{-\gamma} \int_{B_{2^{-\ell}R_0}} u^2 dg
\\
& = (K 2^{-\gamma Q})^\ell \omega^\gamma R_0^{\gamma Q} |B_{2^{-\ell}R_0}|^{-\gamma} \int_{B_{2^{-\ell}R_0}} u^2 dg,
\end{align*}
where $\gamma>0$ is arbitrary and $\omega>0$ is a universal constant. Now we choose $\gamma$ so that $K 2^{-\gamma Q}=1$. This choice gives
\[
\int_{B_{R_0}} u^2 dg  \le C R_0^{\gamma Q} |B_{2^{-\ell}R_0}|^{-\gamma} \int_{B_{2^{-\ell}R_0}} u^2 dg\  \longrightarrow\ 0,
\]
as $\ell \to \infty$ since, by assumption, $u$ vanishes to infinite order at $e$. This shows that $u\equiv 0$ in $B_{R_0}$. By the arbitrariness of $R_0<R$ we conclude that $u\equiv 0$ in $B_R$.

\end{proof}

\begin{dfn}[Unique continuation property]\label{D:sucp}
We say that the operator $\Delta_H$ has the \emph{strong unique continuation property (sucp)}, if $u\equiv 0$ is the only harmonic function that can vanish to infinite order according to Definition \ref{D:infty} in a connected open set $\Om\subset \bG$. We say that $\Delta_H$ has the \emph{weak unique continuation property (wucp)}, if $u\equiv 0$ is the only harmonic function which can vanish in an open subset $\omega\subset \Om$ of a connected open set $\Om\subset \bG$.  
\end{dfn}

\begin{rmrk}[A smooth function with unbounded frequency]\label{R:uf}
Here, we provide an example of a smooth function in a Carnot group $\bG$ with unbounded frequency. For $\ve>0$, let 
\[
u(g) = \begin{cases}
\exp\left(-\frac{1}{\rho(g)^\ve}\right), \ \ \  \ g\in \bG\setminus \{e\},
\\
0,\ \ \ \ \ \ \ \ \ g = e.
\end{cases}
\] 
The function $u\in C^\infty(\bG)$ solves the equation 
\begin{equation}\label{nhs}
\Delta_H u = F,
\end{equation}
where with $f(t) =  \exp\left(-\frac{1}{t^\ve}\right)$ if $t\not= 0$, and $f(0) = 0$, we have let, see Proposition \ref{P:radial},
\[
F = |\nh \rho|^2\left\{f''(\rho) + \frac{Q-1}{\rho} f'(\rho)\right\}.
\]
The natural frequency associated with the equation \eqref{nhs} is $N(r) = r I(r)/H(r)$, where $H(r)$ is defined as in \eqref{H} above, whereas
\[
I(r) = \int_{B_r} \left(|\nh u|^2 + F u\right) dg.
\]
Similarly to Proposition \ref{P:beauty} we can prove
\begin{equation}\label{I}
I(r)  = \int_{S_r} u \frac{Zu}{r}
 |\nabla_H\rho| d\sigma_H.
\end{equation}
Since $\rho$ in \eqref{rho} is homogeneous of degree one, we obtain from the chain rule $Zu = f'(\rho) Z\rho = f'(\rho) \rho$. We thus find from \eqref{I}
\[
I(r) = \frac{\ve}{r^{\ve+1}}  \exp\left(-\frac{2}{r^\ve}\right) \int_{S_r} |\nh \rho| d\sigma_H,
\]
whereas one easily sees that
\[
H(r) = \exp\left(-\frac{2}{r^\ve}\right) \int_{S_r} |\nh \rho| d\sigma_H.
\]
In conclusion,
\[
N(r) = \frac{rI(r)}{H(r)} = \frac{\ve}{r^{\ve}},
\]
which shows that for no $R_0>0$ one can possibly have $N\in L^\infty(0,R_0)$. The unboundedness of the frequency in this example is not surprising since if for some $R_0>0$ we had $N\in L^\infty(0,R_0)$, then from the equation
\[
\frac{d}{dr} \log \frac{H(t)}{t^{Q-1}} = 2 \frac{N(t)}{t},
\]
which is still presently valid with $N(t) = t I(t)/H(t)$, proceeding as in the proof of Theorem \ref{T:Nb} we would obtain a doubling condition as in \eqref{DD} above. But this contradicts the fact that $u$ vanishes to infinite order at $e$, without being identically equal to zero in a full neighborhood of $e$, see Theorem \ref{T:sucp}. If we write
\begin{equation}\label{V}
V = \frac{F}{u} = |\nh \rho|^2\frac{f''(\rho) + \frac{Q-1}{\rho} f'(\rho)}{f(\rho)} = \frac{\ve^2}{\rho^{2+2\ve}}\left(1+\frac{Q-2-\ve}{\ve} \rho^\ve\right) |\nh \rho|^2,
\end{equation}
then this example also shows that, as far as the \emph{sucp} is concerned, the \emph{inverse square potential} condition
\[
|V| \le \frac{C}{\rho^2}|\nh \rho|^2,
\]
is sharp for the operator $-\Delta_H + V$. The function $u$ is in fact a solution of $\Delta_H u = V u$, with a $V$, given by \eqref{V}, that barely violates this condition. We also note that this example shows that, for any $1\le p<Q/2$ there is a potential $V\in L^p_{loc}(\bG)$ for which the \emph{sucp} fails for solutions of $\Delta_H u = V u$. It suffices in fact to choose $0<\ve < \frac{Q}{2p} -1$ in \eqref{V}. For more surprising negative results in this direction the reader should consult the paper \cite{Ba}, whereas for some positive results in the framework of the Heisenberg group $\Hn$ the reader should see \cite{GL}.
  
\end{rmrk}

Clearly, the \emph{sucp} implies the \emph{wucp}. Remarkably, the question of the \emph{sucp}, or even that of the \emph{wucp},  for sub-Laplacians on Carnot groups is a fundamental  question which is largely open. In the Heisenberg group $\Hn$, or more in general in any group of Heisenberg type, it is known that any sub-Laplacian is real analytic hypoelliptic (this follows from the fact, respectively established in \cite{F1} and \cite{K}, that their fundamental solutions are real analytic outside of the singularity), and therefore harmonic functions cannot vanish to infinite order at one point (in the standard sense of vanishing with derivatives of all orders) of a connected open set, unless they vanish identically. However, even for $\Hn$ a quantitative proof of the \emph{sucp} for harmonic functions which does not use real analyticity is not presently known. Thus, for instance, it is still not known whether, without using the real analyticity, harmonic functions in $\Hn$ fall under the scope of Theorem  \ref{T:sucp} above, i.e., they have locally uniformly bounded frequency. 

In the framework of Carnot groups of step two, Helffer had proved in \cite{He} that if a sub-Laplacian in $\bG$ is real-analytic hypoelliptic, then $\bG$ must be a Metivier group. We recall that a Carnot group of step $2$ is said a Metivier group if the Kaplan mapping is non-degenerate, see Definition \ref{D:Metivier} below. Thanks to a result of Metivier \cite{Me2}, see also \cite{Me1}, it is known that if $\bG$ is a Metivier group, then a sub-Laplacian $\Delta_H$ is real-analytic hypoelliptic if and only if it is hypoelliptic. Since by H\"ormander's theorem every sub-Laplacian is hypoelliptic, it follows than in Metivier groups all sub-Laplacians are real-analytic hypoellitpic, and therefore they possess  the standard \emph{sucp}.  However, outside groups of step $2$ there is (to the best of  the authors' knowledge) no known example of a real-analytic sub-Laplacian. It is plausible (for this conjecture see L. Rothschild in \cite{Ro}) that no sub-Laplacian on a Carnot group which is not a Metivier group should be real-analytic hypoelliptic, but this question seems to be open. 

In view of the above considerations, the problem of the boundedness of the frequency, and the deeply connected \emph{sucp} according to Definition \ref{D:sucp}, acquire a fundamental relevance, especially for groups which are not of Metivier type. In Section \ref{S:energy}
 we present some partial progress in this direction.

\section{Global consequence of the boundedness of the frequency}\label{S:global}

In this section we turn to the question of understanding what happens when the frequency $r\to N(u,r)$ of an entire (harmonic) function $u$ on a Carnot group $\bG$ is globally bounded, i.e., it is bounded on $(0,\infty)$. In the special setting of the Heisenberg group $\Hn$ this question has been recently studied in the paper \cite{LTY}, and our Theorem \ref{T:bf} generalizes Theorem 1.1 in their paper to all Carnot groups, except that our proof is much shorter. 
For the definition of stratified polynomials in a Carnot group we refer the reader to \cite{FS} and \cite{BLU}.

\begin{thrm}\label{T:bf}
Let $\bG$ be a Carnot group and suppose that $\Delta_H u = 0$ in $\bG$. If there exists $\kappa \ge 0$ such that $N(r) = N(u,r)\le \kappa$ for every $r>0$, then $u$ is a stratified harmonic polynomial of degree $\ell \le [\kappa]$, with $[\kappa]$ equal to the integral part of $\kappa$ (when $\kappa>0$, we assume that for every $r>0$, $u\not\equiv 0$ in $B_r$).
\end{thrm}

\begin{proof}
If $\kappa = 0$, then $u\equiv$ const. on $\bG$ and we are done. Suppose $\kappa >0$. By Lemma \ref{L:H'} we know that
\[
\frac{d}{dt} \log \frac{H(t)}{t^{Q-1}} = 2 \frac{N(t)}{t} \le \frac{2\kappa}{t},\ \ \ \ 0<t<\infty.
\]
Integrating in $t\in (r,R)$ we find
\begin{equation}\label{gro}
H(R)\le H(r) \left(\frac Rr\right)^{Q-1+2\kappa},\ \ \ \ 0<r<R<\infty.
\end{equation}
Using Federer's co-area formula we easily obtain from \eqref{gro}
\begin{equation}\label{gro2}
\int_{B_R} u^2 |\nh \rho|^2 dg \le \left(\frac Rr\right)^{Q+2\kappa} \int_{B_r} u^2 |\nh \rho|^2 dg,\ \ \ \ 0<r<R<\infty.
\end{equation}
This gives for any $R>1$
\begin{align}\label{gro3}
& \int_{B_R} u^2 |\nh \rho|^2 dg \le R^{Q+2\kappa} \int_{B_1} u^2 |\nh \rho|^2 dg
\\
& \le C R^{Q+2\kappa} ||u||^2_{L^\infty(B_1)}.
\notag
\end{align}
At this point we use Lemma \ref{L:trick} above. For a fixed $R>0$, let $g\in \overline B_R$ such that
\[
\underset{B_R}{\sup}\ |u| = |u(g)|.
\]
Using the harmonicity of $u$, and arguing as in the proof of Theorem \ref{T:Nb} above, we conclude that
\begin{equation}\label{wow}
\underset{B_R}{\sup}\ |u|  \le C^2 M_{\Lambda^2 R}(|u|)(e) \le C^\star \left(R^{-Q} \int_{B_{\Lambda^2 R}} u^2 |\nabla_H \rho|^2 dg\right)^{1/2},
\end{equation}
where $\Lambda>0$ is the universal constant in Lemma \ref{L:trick}.
From \eqref{gro3} and \eqref{wow} we finally have for a constant $\tilde C = \tilde C(\bG,\kappa)>0$
\begin{equation}\label{crucialbound}
||u||_{L^\infty(B_R)} \le \tilde C R^\kappa ||u||^2_{L^\infty(B_1)},\ \ \ \ \text{for every}\ R>\Lambda^{-2}.
\end{equation}
The desired conclusion now follows from \eqref{crucialbound}, and from the asymptotic Liouville Theorem 5.8.8 in \cite{BLU}.

\end{proof}

\section{First variation of the energy and discrepancy}\label{S:energy}


We next want to study further properties of $N(r)$. With this objective in mind it will be important to
compute the first variation $D'(r)$ of the energy. We will use the following
Rellich type identity established in Corollary 3.3 in \cite{GV}.

\begin{prop}\label{P:Rellich}
Let $\bG$ be a Carnot group of arbitrary step, and suppose that $\Om\subset \bG$ is a bounded open set of class $C^1$. For $u\in \Gamma^{2}(\overline{\Om})$ ($\Gamma^2$ is the Folland-Stein class, see \cite{F2}) one has
\begin{align*}
&  2 \int_{\partial{\Om}} \zeta u <\nh u, \bN_H> dH_{N-1}
 + \int_\Om  \emph{div}_{\bG} \zeta |\nh u|^2 dg
\\
&  - 2 \sum_{i=1}^m \int_\Om X_iu [X_i,\zeta]u dg 
 - 2 \int_\Om \zeta u \Delta_H u dg
\notag\\
& = \int_{\partial{\Om}} |\nh u|^2 <\zeta,\nu>
dH_{N-1}, \notag
\end{align*}
where $\zeta$ is a $C^1$ vector field on $\bG$, and $ \bN_H$ is given by \eqref{Nh}.
\end{prop}

We will need the following elementary facts established in \cite{DG}.

\begin{lemma}\label{L:Zprop}
In a Carnot group $\bG$ the infinitesimal generator of group dilations $Z$ enjoys the following properties:
\begin{itemize}
\item[(i)]
\emph{div}$_\bG Z \equiv Q$.
\item[(ii)]
One has $[X_i,Z]=X_i, \quad i=1,...,m,$
\item[(iii)]
$\Delta_H(Zu) = Z(\Delta_H u) + 2 \Delta_Hu$, \emph{for any} $u\in C^{\infty}(\bG)$.
In particular, $Zu$ is harmonic if such is $u$.
\end{itemize}
\end{lemma}

\begin{prop}[First variation of the energy]\label{P:D'}
Let $\bG$ be a Carnot group of arbitrary step, and let $u$ be harmonic in $B_R$. Then, for every $0<r<R$ one has for a.e. $r\in (0,R)$,
\begin{align}\label{rellichBr}
D'(r) & = \frac{Q-2}{r} D(r) + 2 \int_{S_r} \frac{Z
u}{r} <\nh u, \nh \rho> \frac{1}{|\nabla \rho|}dH_{N-1}
\\
& = \frac{Q-2}{r} D(r) + 2 \int_{S_r} \frac{Z u}{r}
<\nh u, \bN_H> dH_{N-1}. 
\notag
\end{align}
\end{prop}

\begin{proof}
Let us begin with observing that the co-area formula gives
\[
D(r) = \int_0^r \int_{S_t} \frac{ \uh}{|\nabla \rho|} dH_{N-1} dt .
\]
This identity implies for a.e. $r\in (0,R)$
\[
D'(r) = \int_{S_r}  \frac{\uh}{|\nabla \rho|} dH_{N-1}.
\]
We now apply Proposition \ref{P:Rellich}, in which we take $u$ such that $\Delta_Hu=0$,
$\zeta = Z$, and $\Om = B_r$. Using Lemma \ref{L:Zprop}, which gives
div$_{\bG} Z= Q$, that $[X_i,Z]=X_i$, and that $Z \rho = \rho = r$
on $S_r$, we obtain
\begin{equation*}\label{rellich2}
r \int_{S_r} \frac{|\nh u|^2}{|\nabla \rho|}
dH_{N-1} = 2 \int_{S_r} Z u <\nh u, \bN_H> dH_{N-1}
 + (Q-2) \int_{B_r}  |\nh u|^2 dg.
\end{equation*}
This formula gives the desired conclusion.

\end{proof}


\begin{dfn}[Discrepancy]\label{D:discrep}
Given a function $u$ in a Carnot group $\bG$, we define the \emph{discrepancy} of $u$ at $e\in \bG$ as
\begin{equation}\label{E}
E_u = <\nh u,\nh \rho> - \frac{Zu}{\rho} |\nh \rho|^2.
\end{equation} 
The discrepancy of $u$ at any other $g_0\in \bG$ is defined as the discrepancy at $e$ of the function $v(g) = u(g_0^{-1} \circ g)$. 
\end{dfn}

We note explicitly that, using \eqref{E}, we can rewrite the first variation formula \eqref{rellichBr} in Proposition \ref{P:D'} as follows
\begin{equation}\label{D'2}
D'(r) = \frac{Q-2}{r} D(r) + 2 \int_{S_r} \left(\frac{Z
u}{r}\right)^2 \frac{|\nh \rho|^2}{|\nabla \rho|}dH_{N-1} + 2 \int_{S_r} \left(\frac{Z
u}{r}\right) E_u \frac{dH_{N-1}}{|\nabla \rho|}. 
\end{equation}

\begin{thrm}\label{T:D'gen}
Let $\bG$ be a Carnot group, and let $u$ be harmonic in $B_R$. Suppose in addition that $u$ has vanishing discrepancy at $e$ in $B_R$, i.e.,
\begin{equation}\label{aa}
E_u = <\nh u,\nh \rho> - \frac{Zu}{\rho} |\nh \rho|^2 \equiv 0.
\end{equation}
Then, for every $0<r<R$ one has for a.e. $r\in (0,R)$,
\begin{equation}\label{D'gen}
D'(r)  = \frac{Q-2}{r} D(r) + 2 \int_{S_r} \left(\frac{Z
u}{r}\right)^2 |\nh \rho| d\sigma_H.
\end{equation}
\end{thrm}

\begin{proof}
If \eqref{aa} holds, then the desired conclusion follows immediately from \eqref{D'2}. 

\end{proof}

While in Section \ref{S:examples} we will provide large classes of  examples in which  the assumption \eqref{aa} is fulfilled, we state here a simple lemma which shows that the class of functions satisfying \eqref{aa} is not empty.

\begin{prop}\label{P:radzero}
In a Carnot group $\bG$ let $u(g) = f(\rho(g))$ for some $C^1$ function $f:[0,\infty)\to \R$. Then, $u$ has vanishing discrepancy $E_u$.
\end{prop}

\begin{proof}
By the chain rule we have $\nh u = f'(\rho) \nh \rho$, $Zu = f'(\rho) Z\rho = f'(\rho) \rho$, where we have used the fact that $\rho$ is homogenous of degree one. Thus,
\[
E_u = f'(\rho) |\nh \rho|^2 - f'(\rho) |\nh \rho|^2 = 0.
\]

\end{proof}

\begin{thrm}\label{T:mono}
Let $\bG$ be a Carnot group, and let $u$ be harmonic in $B_R$. Suppose in addition that $u$ has vanishing discrepancy, i.e., $u$ satisfies the differential equation \eqref{aa}
in $B_R$. Then, the frequency of $u$ is non-decreasing on $(0,R)$. In particular, $N(u,\cdot)\in L^\infty(0,R)$. In view of Theorems \ref{T:Nb} and \ref{T:sucp}
we conclude that harmonic functions of vanishing discrepancy have the \emph{sucp} in $\bG$.
\end{thrm}

\begin{proof}
From the definition \eqref{N} of $N(u,r) = N(r)$, we have for a.e. $r\in (0,R)$,
\begin{align}\label{N'}
\frac{d}{dr} \log N(r) & = \frac{1}{r} + \frac{D'(r)}{D(r)} - \frac{H'(r)}{H(r)}
\\
& = \frac{1}{r} + \frac{Q-2}{r} + 2 \frac{\int_{S_r} \left(\frac{Z
u}{r}\right)^2 |\nh \rho| d\sigma_H}{D(r)} - \frac{Q-1}{r} - 2 \frac{D(r)}{H(r)}
\notag\\
& = 2\ \frac{\int_{S_r} \left(\frac{Z
u}{r}\right)^2 |\nh \rho| d\sigma_H}{\int_{S_r} u \frac{Zu}{r}
 |\nabla_H\rho| d\sigma_H} - 2\ \frac{\int_{S_r} u \frac{Zu}{r}
 |\nabla_H\rho| d\sigma_H}{\int_{S_r} u^2
 |\nabla_H\rho| d\sigma_H},
\notag
\end{align}
where we have used, in the order, \eqref{D'gen}, Lemma \ref{L:H'} and Proposition \ref{P:beauty}.
From \eqref{N'} we immediately see that $\frac{d}{dr} \log N(r)\ge 0$ follows from Cauchy-Schwarz inequality.

\end{proof}

Theorem \ref{T:mono} immediately implies the following partial converse to Proposition \ref{P:hom}.

\begin{prop}\label{P:monhom}
Let $\Delta_H u = 0$ in $B_R\subset \bG$, and suppose that $u$ has vanishing discrepancy at $e$ in $B_R$, i.e., \eqref{aa} holds. If there exist $\kappa\ge 0$ such that
\[
N(u,\cdot)\equiv \kappa\ \ \ \ \ \ \text{in}\ (0,R).
\]
Then, $u$ is homogeneous of degree $\kappa$ in $B_R$.
\end{prop}

\begin{proof}
We write for simplicity $N(r)$ instead of $N(u,r)$. From \eqref{N'} in Theorem \ref{T:mono} we have 
\begin{equation}\label{cs}
\frac{d}{dr} \log N(r) = 2\ \frac{\int_{S_r} \left(\frac{Z
u}{r}\right)^2 |\nh \rho| d\sigma_H}{\int_{S_r} u \frac{Zu}{r}
 |\nabla_H\rho| d\sigma_H} - 2\ \frac{\int_{S_r} u \frac{Zu}{r}
 |\nabla_H\rho| d\sigma_H}{\int_{S_r} u^2
 |\nabla_H\rho| d\sigma_H} \ge 0
\end{equation}
by Cauchy-Schwarz inequality. But if we assume that $N(r) \equiv \kappa$ in $(0,R)$, we conclude from \eqref{cs} that for $0<r<R$ one has
\[
\frac{\int_{S_r} \left(\frac{Z
u}{r}\right)^2 |\nh \rho| d\sigma_H}{\int_{S_r} u \frac{Zu}{r}
 |\nabla_H\rho| d\sigma_H} -  \frac{\int_{S_r} u \frac{Zu}{r}
 |\nabla_H\rho| d\sigma_H}{\int_{S_r} u^2
 |\nabla_H\rho| d\sigma_H} \equiv 0,
\]
and therefore we must have equality in the Cauchy-Schwarz inequality. This implies the existence of a function $\la(r)$, such that for every $0<r<R$ one has
\[
\frac{Zu}{r} =  \la(r) u\ \ \ \ \ \ \ \text{on}\ S_r.
\]
Using this information in Proposition \ref{P:beauty}, we obtain
\[
D(r)  = \int_{S_r} u \frac{Zu}{r}
 |\nabla_H\rho| d\sigma_H = \la(r) \int_{S_r} u^2
 |\nabla_H\rho| d\sigma_H = \la(r) H(r).
\]
In turn, this gives
\[
\kappa \equiv N(r) \equiv r \la(r),
\]
which forces the conclusion $\la(r) = \frac{\kappa}{r}$. We thus infer that $Zu = \kappa u$, and thus $u$ is homogeneous of degree $\kappa$. 

\end{proof}

\section{Two theorems on the bounded frequency}\label{S:2theo}

We emphasize that we do not know whether the monotonicity of the frequency asserted in Theorem \ref{T:mono} is true without the additional assumption \eqref{aa}. We strongly suspect that, unlike what happens for classical harmonic functions, see \cite{A}, if the discrepancy of $u$ does not vanish identically, the pure monotonicity of the frequency fails. However, as we have shown in Theorem \ref{T:Nb} and Theorem \ref{T:sucp}, its boundedness would suffice to establish the \emph{sucp}. Although it is very tempting to conjecture that, for a given a harmonic function in $B_R\subset \bG$, its frequency is always locally bounded, presently we do not know whether this fundamental property is true.  The next two results provide some interesting progress toward this open question. The former states that in every group of Metivier type the frequency of a harmonic function is in fact locally bounded (for the notion of group of Metivier see Definition \ref{D:Metivier} below). For the Heisenberg group $\Hn$ the proof of this result was suggested to us by Agnid Banerjee. Here, we reproduce with his kind permission a generalization of his idea.

\begin{thrm}\label{T:agnid}
Let $\bG$ be a group of Metivier type, and let $\Delta_H u = 0$ in $B_R\subset \bG$. There exist $C, R_0>0$, depending on $u$ such that for every $0<r<2R_0$ one has
\[
\int_{B_{2r}} u^2 dg \le C \int_{B_{r}} u^2 dg.
\]
As a consequence of this result and of Theorem \ref{T:Nb}, we have $N(u,\cdot)\in L^\infty(0,R_0)$.
\end{thrm}

\begin{proof}
We begin by observing that, thanks to \eqref{alphabeta} above, if for $g = (z,t)\in \bG$ we indicate with $|g| = (|z|^4 + 16 |t|^2)^{1/4}$ the non-isotropic Koranyi gauge on $\bG$, then there exist universal constants $\alpha, \beta>0$ such that 
\[
B^{|\cdot|}_{\beta r} \subset B_r \subset B^{|\cdot|}_{\alpha r}.
\]
In order to prove the theorem it will thus suffice to establish the following inequality
\begin{equation}\label{ma}
\int_{B^{|\cdot|}_{2 \alpha r}} u^2 dg \le C \int_{B^{|\cdot|}_{\beta r}} u^2 dg,
\end{equation}
for some constant $C = C_u>0$ and for all $0<r<R_0$. Next, we observe that using the exponential mapping we can identify $\bG$ with $(\R^N,\circ)$, where $\circ$ denotes the group law of $\bG$. Having done this identification, in what follows, given a point $g\in \R^N$, we will denote by $B^e(g,r)$ the standard Euclidean ball centered at $g$ with radius $r$, whereas we will continue indicating with $B_r(g)$ the (pseudo-) ball defined in \eqref{balls}, \eqref{br}. As a consequence of the result of Metivier \cite{Me2}, cited at the end of Section \ref{S:fv}, we know that $u$ is real-analytic in $B_R$. We will appeal to the following theorem due to P. Garrett and the first named author, see Theorem $1$ in \cite{GG}:
\emph{There exist $R_0>0$ and $C_{N,u}>0$ such that for every $g\in B_{R/2}$ and every $0<r<R_0$, one has}
\begin{equation}\label{gg}
\int_{B^e(g,2r)} u^2 dg' \le C_{N,u} \int_{B^e(g,r)} u^2 dg'.
\end{equation}
In what follows, if $\kappa>0$ and $B^e$ indicates a Euclidean ball $B^e(g,r)\subset \R^N$, we will denote by $\kappa B^e$ the concentric ball $B^e(g,\kappa r)$. Using the Besicovitch covering lemma in $\R^N$, we can find a constant $K>1$ depending only on $\alpha, \beta$, such that we can cover $B^{|\cdot|}_\beta$ and $B^{|\cdot|}_{2\alpha}$ with countable families of Euclidean balls $\{B^e_{j}\}$ and $\{KB^e_{j}\}$ respectively,  such that the  following property holds:
there exists $\mathcal N_N \in \mathbb N$ such that 
\begin{equation}\label{besi}
B^{|\cdot|}_{\beta} \subset \bigcup_{i=1,...,\mathcal N_N} \bigcup_{B^e_{j} \in \mathcal G_{i}} B^e_{j},\ \ \  \ B^{|\cdot|}_{2\alpha} \subset \bigcup_{i=1,...,\mathcal N_N} \bigcup_{B^e_{j} \in \mathcal G_{i}} KB^e_{j}
\end{equation}
where each $\mathcal G_{i}$ is a countable disjoint collection of balls. Consider now the Euclidean dilations $T_r: R^{N} \to R^{N}$ defined by $T_r(g)= r g = (rz,rt)$. If $B^e_j = B(g_j,r_j)$, then we clearly have  $T_r(B^e_{j})= B^e(T_r(g_j),r r_j)$. We claim that  
\[
B^{|\cdot|}_{\beta r} \subset T_r(B^{|\cdot|}_\beta),\ \ \ \ \ \text{for}\ r\le 1,
\]
or, equivalently, $T_r^{-1} (B^{|\cdot|}_{\beta r}) \subset B^{|\cdot|}_\beta$. To see this, let $g = (z,t) \in B^{|\cdot|}_{\beta r}$, so that
$(|z|^4 + 16 |t|^2)^{1/4} < \beta r$. Then,  consider $T_r^{-1}(g) = (z/r, t/r)$. When $r \leq  1$ we have 
\[
\left|T_r^{-1}(g)\right| = (\frac{|z|^4}{r^4} + 16 \frac{|t|^2}{r^2})^{1/4} = \frac{1}{r} ( |z|^4 + 16 r^2 |t|^2)^{1/4} \leq \frac{1}{r} ( |z|^4 + 16 |t|^2)^{1/4} < \frac{\beta r}{r}  = \beta.
\]
Hence, the above claim follows. From the claim and \eqref{besi} we thus conclude that 
\begin{equation}\label{1r}
B^{|\cdot|}_{\beta r}  \subset \bigcup_{i=1,...,\mathcal N_N} \bigcup_{B^e_{j} \in \mathcal G_{i}} T_r(B^e_{j}),
\end{equation}
and, analogously,
\begin{equation}\label{2r}
B^{|\cdot|}_{2\alpha r} \subset \bigcup_{i=1,...,\mathcal N_N} \bigcup_{B^e_{j} \in \mathcal G_{i}} T_r(KB^e_{j}).
\end{equation}
Since $\mathcal N_{N}$ is finite and fixed, by \eqref{2r} there exists at least one index $i_0\in \{1,...,\mathcal N_N\}$, such that
\begin{equation}\label{oneway}
\int_{\bigcup_{B^e_{j}\in \mathcal G_{i_0}} T_r(KB^e_{j}) } u^2 dg \geq \frac{1}{\mathcal N_N} \int_{B^{|\cdot|}_{2\alpha r}} u^2 dg.
\end{equation}
On the other hand, using the crucial doubling condition \eqref{gg} on Euclidean balls, and the fact that for $i=1,...,\mathcal N_N$ each $\mathcal G_{i}$ is a disjoint collection,  we obtain
\begin{equation}\label{l2}
\int_{ \bigcup_{B^e_{j}\in \mathcal G_{i_0}} T_r(KB^e_{j}) } u^2 dg \leq C_{N,u} \int_{ \bigcup_{B^e_{j}\in \mathcal G_{i_0}} T_r(B^e_{j})} u^2 dg \leq C_{N,u} \int_{B^{|\cdot|}_{\beta r} } u^2 dg,
\end{equation}
where in the last inequality we have used \eqref{1r}.
Then, by \eqref{oneway} and \eqref{l2} we have
\begin{equation}
\int_{B^{|\cdot|}_{ 2 \alpha r} } u^2 dg \leq \mathcal N_{N} \int_{ \bigcup_{B^e_{j}\in \mathcal G_{i_0}} T_r(K B^e_{j}) } u^2 dg \leq  \mathcal N_N C_{N,u} \int_{B^{|\cdot|}_{\beta r} } u^2 dg,
\end{equation}
which gives the desired conclusion \eqref{ma}.

\end{proof}

\begin{rmrk}\label{R:diss}
Theorem \ref{T:agnid} shows that in every group of Metivier type, and thus in particular in every group of Heisenberg type (see Remark \ref{R:metivier} below), the frequency of a harmonic function is locally bounded. If, on the one hand, this result answers affirmatively a question which was left open in \cite{GL} and provides evidence in favor of a more general conjecture in any Carnot group, it does nonetheless contain an unsatisfactory aspect. The doubling constant in \eqref{ma} depends on the doubling constant on Euclidean balls in the basic estimate \eqref{gg} from \cite{GG}.  The latter, in turn, depends on the real-analytic character of $u$, and thus on the value of all derivatives at $u$ at the base point $e\in \bG$. This is in sharp contrast with the doubling constant which one obtains from a local bound on the frequency which solely  depends on the $L^2$ norm of $u$ and of the first horizontal derivatives of $u$ on a fixed ball.  
\end{rmrk}

Before stating the next result, we introduce some quantities which play a role in its statement. In what follows, given $R>0$, we will consider increasing functions $f:(0,R)\to (0,\infty)$ satisfying the Dini integrability condition
\begin{equation}\label{dini}
\int_0^R f(t) \frac{dt}{t} <\infty.
\end{equation}
We note for future use that \eqref{dini} implies that
\begin{equation}\label{dini2}
\underset{t\to 0^+}{\lim}\ f(t) = 0.
\end{equation}

\begin{thrm}\label{T:GLgen}
Let $u$ be a harmonic function in $B_R$ and suppose that there exists a function $f$ satisfying \eqref{dini} such that the discrepancy of $u$ at $e$ satisfy  
\begin{equation}\label{EE}
|E_u| \le \frac{f(\rho)}{\rho}  |\nh \rho|^2 |u|,\ \ \ \ \text{in}\ B_R.
\end{equation}
Then, there exists $0<R_0<R$ such that $N(u,\cdot) \in L^\infty(0,R_0)$. As a consequence, if $\kappa = ||N(u,\cdot)||_{L^\infty(0,R_0)}$, then $u$ satisfies the doubling condition 
\begin{equation}\label{DDC}
\int_{B_{2r}} u^2 dg \le C^\star 2^{Q+2\kappa}  \int_{B_r} u^2 dg, \ \ \ \ \ 0<r<R_0/2,
\end{equation}
and therefore if $u$ vanishes to infinite order at $e$, we must have $u\equiv 0$ in $B_{R_0}$.
\end{thrm}

\begin{proof}
If we proceed as for \eqref{N'}, but using \eqref{D'2} instead of \eqref{D'gen}, we obtain
\begin{align}\label{NN'}
\frac{d}{dr} \log N(r) & = 2\ \frac{\int_{S_r} \left(\frac{Z
u}{r}\right)^2 |\nh \rho| d\sigma_H}{\int_{S_r} u \frac{Zu}{r}
 |\nabla_H\rho| d\sigma_H} - 2\ \frac{\int_{S_r} u \frac{Zu}{r}
 |\nabla_H\rho| d\sigma_H}{\int_{S_r} u^2
 |\nabla_H\rho| d\sigma_H} + 2\ \frac{\int_{S_r} \left(\frac{Z
u}{r}\right) E_u \frac{dH_{N-1}}{|\nabla \rho|}}{\int_{S_r} u \frac{Zu}{r}
 |\nabla_H\rho| d\sigma_H}.
\end{align}
We now distinguish two cases:
\begin{equation}\label{case1}
\left(\int_{S_r} \left(\frac{Z
u}{r}\right)^2 |\nh \rho| d\sigma_H\right)^{1/2} \left(\int_{S_r} u^2 |\nh \rho| d\sigma_H\right)^{1/2} \le \sqrt 2\ D(r),
\end{equation}
\begin{equation}\label{case2}
\left(\int_{S_r} \left(\frac{Z
u}{r}\right)^2 |\nh \rho| d\sigma_H\right)^{1/2} \left(\int_{S_r} u^2 |\nh \rho| d\sigma_H\right)^{1/2} > \sqrt 2\ D(r).
\end{equation}
If \eqref{case1} occurs we use Cauchy-Schwarz inequality to find
\begin{align*}
& \left|\int_{S_r} \left(\frac{Z
u}{r}\right) E_u \frac{dH_{N-1}}{|\nabla \rho|}\right| \le \left(\int_{S_r} \left(\frac{Z
u}{r}\right)^2 |\nh \rho| d\sigma_H\right)^{1/2} \left(\int_{S_r} \frac{|E_u|^2}{|\nh \rho|^2}  \frac{dH_{N-1}}{|\nabla \rho|}\right)^{1/2}
\\
& \le \frac{f(r)}{r}\ \left(\int_{S_r} \left(\frac{Z
u}{r}\right)^2 |\nh \rho| d\sigma_H\right)^{1/2} \left(\int_{S_r} u^2 |\nh \rho| d\sigma_H\right)^{1/2},
\end{align*}
where in the last inequality we have used the assumption \eqref{EE}. By \eqref{case1} we then find
\begin{align*}
& \left|\int_{S_r} \left(\frac{Z
u}{r}\right) E_u \frac{dH_{N-1}}{|\nabla \rho|}\right| \le 2 \sqrt 2 \frac{f(r)}{r} D(r),
\end{align*} 
and we thus have from \eqref{NN'}
\begin{align}\label{case11}
\frac{d}{dr} \log N(r) & \ge 2\ \frac{\int_{S_r} \left(\frac{Z
u}{r}\right)^2 |\nh \rho| d\sigma_H}{\int_{S_r} u \frac{Zu}{r}
 |\nabla_H\rho| d\sigma_H} - 2\ \frac{\int_{S_r} u \frac{Zu}{r}
 |\nabla_H\rho| d\sigma_H}{\int_{S_r} u^2
 |\nabla_H\rho| d\sigma_H} - 2 \sqrt 2 \frac{f(r)}{r}
 \\
 & \ge - 2 \sqrt 2 \frac{f(r)}{r},
 \notag
\end{align}
by Cauchy-Schwarz inequality. 

If case \eqref{case2} occurs we use the inequality $|ab|\le \frac{a^2+b^2}{2}$ to obtain
\begin{align*}
& 2\ \left|\int_{S_r} \left(\frac{Z
u}{r}\right) E_u \frac{dH_{N-1}}{|\nabla \rho|}\right| \le  \int_{S_r} \left(\frac{Z
u}{r}\right)^2 |\nh \rho| d\sigma_H +  \int_{S_r} \frac{|E_u|^2}{|\nh \rho|^2}  \frac{dH_{N-1}}{|\nabla \rho|}
\\
& \le \int_{S_r} \left(\frac{Z
u}{r}\right)^2 |\nh \rho| d\sigma_H  + \frac{f(r)^2}{r^2}\ \int_{S_r} u^2 |\nh \rho| d\sigma_H,
\end{align*}
where again we have used the assumption \eqref{EE}.
Substituting in \eqref{NN'} we find
\begin{align}\label{case11}
\frac{d}{dr} \log N(r) & \ge \frac{\int_{S_r} \left(\frac{Z
u}{r}\right)^2 |\nh \rho| d\sigma_H}{\int_{S_r} u \frac{Zu}{r}
 |\nabla_H\rho| d\sigma_H} - 2\ \frac{\int_{S_r} u \frac{Zu}{r}
 |\nabla_H\rho| d\sigma_H}{\int_{S_r} u^2
 |\nabla_H\rho| d\sigma_H} - \frac{f(r)^2}{r^2} \frac{H(r)}{D(r)}
 \\
 & \ge  - \frac{f(r)^2}{r^2} \frac{H(r)}{D(r)},
\notag
\end{align}
where in the last inequality we have used \eqref{case2}. 
By \eqref{dini2} we can now choose $R_0\in (0,R)$ such that
\[
f(r) \le 1,\ \ \ \ 0<r<R_0.
\]
From \eqref{case11} we thus obtain 
\begin{equation}\label{NNN}
\frac{d}{dr} \log N(r) \ge - \frac{f(r)}{r^2} \frac{H(r)}{D(r)},\ \ \ \ 0<r<R_0.
\end{equation}

To obtain a favorable estimate from \eqref{NNN}, at this point we consider the set
\[
\Lambda_{R_0} = \{0<r<R_0\mid N(r) > \max\{1,N(R_0)\}.
\] 
Since the function $r\to N(r)$ is absolutely continuous, the set $\Lambda_{R_0}$ is an open set, and therefore  we can write it
\begin{equation}\label{lambda}
\Lambda_{R_0} = \bigcup_{j=1}^\infty (a_j,b_j),\ \ \ \ \ a_j, b_j \not\in \Lambda_{R_0}.
\end{equation}
We stress that if $r\in \Lambda_{R_0}$, we have $N(r) > 1$, and therefore
\[
\frac{H(r)}{D(r)} < r.
\]
We thus conclude from \eqref{NNN} that
\begin{equation}\label{NNN'}
\frac{d}{dr} \log N(r) \ge - \frac{f(r)}{r},\ \ \ \ \ \ r\in \Lambda_{R_0}.
\end{equation}
Combining \eqref{case11} and \eqref{NNN'} we obtain the following crucial information
 \begin{equation}\label{Nfinal}
\frac{d}{dr} \log N(r) \ge - 2 \sqrt 2 \frac{f(r)}{r},\ \ \ \ \ \ r\in \Lambda_{R_0}.
\end{equation}
With the estimate \eqref{Nfinal} in hands we can now prove that, in fact, $N\in L^\infty(0,R_0)$. 
We proceed as follows. We define 
\[
J_r=\{ t\in (r,2r)\mid t\not\in \Lambda_{R_0}\}.
\]
On $J_r$ we have trivially
\[
0 \le N(t) \le \max\{N(R_0),1\}.
\]
On the other hand, integrating the inequality in \eqref{Nfinal} on $(r,b_j)$, where $r\in (a_j,b_j)$ and $(a_j,b_j)$ is one of the intervals in the decomposition \eqref{lambda}, we obtain
\[
\log \frac{N(b_j)}{N(r)} = \int_r^{b_j} \frac{d}{dt} \log N(t) dt \ge - 2\sqrt 2 \int_0^{R_0} f(t) \frac{dt}{t}  =  -  F_0,
\]
where we have let
\[
F_0 = 2\sqrt 2 \int_0^{R_0} f(t) \frac{dt}{t}  <\infty.
\]
Recalling that $b_j\not\in \Lambda_{R_0}$, from this inequality we find for every $r\in \Lambda_{R_0}$
\[
N(r) \le \exp(F_0) \max\{N(R_0),1\}.
\]
We have thus shown that $N\in L^\infty(0,R_0)$, with 
\begin{equation}\label{kappaN}
||N||_{L^\infty(0,R_0)} \le \kappa = \max\{1,\exp(F_0)\}\ \max\{N(R_0),1\}.
\end{equation}
By Theorem \ref{T:Nb} we obtain \eqref{DDC} with $\kappa$ given by \eqref{kappaN}.

\end{proof}


\section{One-parameter Weiss type monotonicity formulas on Carnot groups}\label{S:WM}

In this section we establish a new monotonicity formula for the subclass of harmonic functions which have vanishing discrepancy at $e$. Our main result is inspired to a monotonicity formula originally proved by G. Weiss \cite{W} for studying the classical obstacle problem. In their recent work \cite{GP},  Petrosyan and the first named author discovered a one-parameter family of Weiss type monotonicity formulas, satisfied by the solutions of the lower-dimensional obstacle problem, and which play a key role in the analysis of the so-called blow-ups at singular points. These authors proved that such family of monotonicity formulas are deeply connected to Almgren's monotonicity of the frequency. The following theorem is in a similar spirit and we hope that it will prove useful in the analysis of variational inequalities of obstacle type.

\begin{thrm}\label{T:weissG}
Let $\bG$ be a Carnot group, and let $u$ be harmonic in $B_R$. For every $\kappa \ge 0$ define for $0<r<R$
\begin{equation}\label{W}
\mathcal W_\kappa(u,r) \overset{def}{=} \frac{1}{r^{Q-2+2\kappa}} \int_{B_r} |\nabla_H u|^2 dg - \frac{\kappa}{r^{Q-1+2\kappa}} \int_{S_r} u^2 |\nabla_H\rho| d\sigma_H.
\end{equation}
If $u$ has vanishing discrepancy at $e$ in $B_R$, i.e., $u$ satisfies \eqref{aa} above, then one has
\begin{equation}\label{W'}
\frac{d}{dr} \mathcal W_\kappa(u,r) = \frac{2}{r^{Q+2\kappa}} \int_{S_r} \left(Zu - \kappa u\right)^2 |\nabla_H\rho| d\sigma_H.
\end{equation}
As a consequence, $r\to \mathcal W_\kappa(u,r)$ is non-decreasing. Furthermore, $\mathcal W_\kappa(u,\cdot)$ is constant if and only if $u$ is homogeneous of degree $\kappa$.
\end{thrm}  

\begin{proof}
From \eqref{W} it is clear that we can write
\[
\mathcal W_\kappa(u,r) \overset{def}{=} \frac{1}{r^{Q-2+2\kappa}} D(r) - \frac{\kappa}{r^{Q-1+2\kappa}} H(r).
\]
Applying Lemma \ref{L:H'} and the identity \eqref{D'2}, we thus obtain 
\begin{align*}
\frac{d}{dr} \mathcal W_\kappa(u,r) & =  \frac{1}{r^{Q-2+2\kappa}} D'(r) - \frac{Q - 2 + 2\kappa}{r} \frac{D(r)}{r^{Q-2+2\kappa}} + \frac{\kappa(Q-1+2\kappa)}{r} \frac{H(r)}{r^{Q-1+2\kappa}} - \frac{\kappa}{r^{Q-1+2\kappa}} H'(r)
\\
& = \frac{1}{r^{Q-2+2\kappa}}\left\{D'(r) - \frac{Q - 2 + 2\kappa}{r} D(r)- \frac{\kappa}{r} H'(r) + \frac{\kappa(Q-1+2\kappa)}{r^2} H(r)\right\}
\\
& = \frac{1}{r^{Q-2+2\kappa}}\bigg\{ \frac{Q-2}{r} D(r) + 2 \int_{S_r} \left(\frac{Z u}{r}\right)^2 |\nh \rho|
d\sigma_H  - \frac{Q - 2 + 2\kappa}{r} D(r)
\\
& - \frac{\kappa}{r}\left(\frac{Q-1}{r} H(r) + 2
D(r)\right) + \frac{\kappa(Q-1+2\kappa)}{r^2} H(r) + 2  \int_{S_r} \left(\frac{Z
u}{r}\right) E_u \frac{dH_{N-1}}{|\nabla \rho|}\bigg\}
\\
& = \frac{2}{r^{Q-2+2\kappa}} \left\{\int_{S_r} \left(\frac{Z u}{r}\right)^2 |\nh \rho|
d\sigma_H  - \frac{2\kappa}{r} \int_{S_r} u \frac{Zu}{r}
 |\nabla_H\rho| d\sigma_H
  + \frac{\kappa^2}{r^2} \int_{S_r} u^2 |\nabla_H\rho| d\sigma_H \right\}
\\
& = \frac{2}{r^{Q+2\kappa}} \int_{S_r} \left(Zu - \kappa u\right)^2   |\nabla_H\rho| d\sigma_H  +  \frac{2}{r^{Q-2+2\kappa}} \int_{S_r} \left(\frac{Z
u}{r}\right) E_u \frac{dH_{N-1}}{|\nabla \rho|}.
\end{align*}
It is now clear that if $u$ satisfies \eqref{aa}, then the identity \eqref{W'} holds.

\end{proof}

Before proving the next result we recall the Cauchy-Schauder estimates established in \cite{CDG}.
Although such estimates are valid for general H\"ormander type operators, we state them in the special setting of Carnot groups.

\begin{lemma}\label{L:Harmonic}
Let $\bG$ be a Carnot group, and suppose that
$\Delta_H u = 0$ in $B_{4r}(g)$. For any $s\in \mathbb
N$
\[
|X_{j_1}X_{j_2}...X_{j_s}u(g)| \leq \frac{C}{r^s}
\underset{\overline{B}_r(g)}{\max}\ |u|,
\]
for some universal constant $C=C(\bG,s)>0$. In the above estimate, for every
$i = 1,...,s,$ the index $j_i$ runs in the set $\{1,...,m\}$ .
\end{lemma}

\begin{dfn}\label{D:sh} Given a Carnot group $\bG$ with a sub-Laplacian $\Delta_H$, and a nonnegative integer $\kappa$, we indicate with $\mathfrak P_\kappa(\bG)$ the (finite-dimensional) space of all \emph{stratified solid harmonics} of degree $\kappa$. I.e., $\mathfrak P_\kappa(\bG)$ is the space of all stratified polynomials $P_\kappa$ on $\bG$ which are $\delta_\lambda$-homogeneous of degree $\kappa$ (recall \eqref{dilG}), and such that $\Delta_H P_\kappa = 0$. 
\end{dfn}

We have the following global (partial) converse to Proposition \ref{P:hom}, see also Proposition \ref{P:monhom} above.

\begin{thrm}\label{T:N=k}
Let $u$ be a harmonic function in $\bG$ satisfying the additional hypothesis \eqref{aa}, and assume that for no $R>0$ we have $u\equiv 0$ in $B_R$. If for some number $\kappa \ge 0$ we have $N(u,r) \equiv \kappa$, then $\kappa$ must be a nonnegative integer and $u = P_\kappa$, with $P_\kappa \in \mathfrak P_\kappa(\bG)$ (see Definition \ref{D:sh}).
\end{thrm}

\begin{proof}

From the hypothesis that for every $r>0$ one has $u\not\equiv 0$ in $B_r$ and Lemma \ref{L:nondeg} we now that $H(r) \not= 0$ for every $r>0$. We can thus divide by $H(r)$ in \eqref{W}, obtaining
\begin{equation}\label{Wk}
\mathcal W_\kappa(u,r) = \frac{H(r)}{r^{Q-1+2\kappa}} \big(N(r) - \kappa\big).
\end{equation}
By the hypothesis that $N(r) \equiv \kappa$ we conclude that $\mathcal W_\kappa \equiv 0$.
But then, Theorem \ref{T:weissG} gives 
\[
\int_{S_r} \left(Zu - \kappa u\right)^2 |\nabla_H\rho| d\sigma_H = 0,\ \ \ \  \ \ 0<r<\infty.
\]
By the co-area formula we conclude
\[
\int_{\bG} \left(Zu - \kappa u\right)^2 |\nabla_H\rho|^2 dg = 0.
\]
In particular, we must have 
\begin{equation}\label{Zuu}
Zu = \kappa u,\ \ \ \ \ \text{in}\ \bG.
\end{equation}
This implies for any $g\in \bG$ such that $\rho(g)\ge 1$
\[
u(g) = \rho(g)^\kappa u(\delta_{\rho(g)^{-1}}g),
\]
and therefore
\[
|u(g)| \le \left(\underset{\p B_1}{\max}\ |u|\right)  \rho(g)^\kappa.
\]
This estimate gives for some constant $M = M(u)>0$,
\begin{equation}\label{growthkappa}
|u(g)| \le M (1+  \rho(g))^\kappa,\ \ \ \ g\in \bG.
\end{equation}
Using \eqref{growthkappa}, we now invoke the Cauchy-Schauder estimates in Lemma \ref{L:Harmonic}, obtaining for every $s>\kappa$ and for every $g\in \bG$
\[
|X_{j_1}X_{j_2}...X_{j_s}u(g)|\ \leq\ \frac{C}{r^{s-\kappa}}.
\]
Letting $r\to \infty$ we conclude that all derivatives of $u$ of degree $s>\kappa$ must vanish in $\bG$. By the stratified Taylor's formula in \cite{FS}, we conclude that $\kappa$ is an integer and $u$ is a stratified harmonic polynomial of order $\kappa$, i.e., an element of the space $\mathfrak P_\kappa(\bG)$.

\end{proof}

\begin{rmrk}\label{R:hom}
We emphasize that \eqref{Zuu} above, i.e., the homogeneity of $u$, could have also been directly deduced from Proposition \ref{P:monhom} above. This is not surprising since Theorem \ref{T:weissG} represents a rescaled quantitative version of Proposition \ref{P:monhom}, with the intertwining between one result and the other given by the identity \eqref{Wk}. 
\end{rmrk}


\section{Analysis of the discrepancy in Carnot groups of Heisenberg type}\label{S:examples}

In this section we analyze in more detail, in the setting of groups of Heisenberg type, the notion of discrepancy introduced in Definition \ref{D:discrep} above. Throughout this section we assume that $\bG$ is a Carnot group of step two, with Lie algebra $\bg = V_1\oplus V_2$, where $[V_1,V_1] = V_2, [V_1,V_2] = \{0\}$. We endow $\bg$ with an inner product with respect to which $\{e_1,...,e_m\}$ and $\{\ve_1,...,\ve_k\}$ denote an orthonormal basis of $V_1$ and  $V_2$, respectively. If $X_i(g) = dL_g(e_i)$, $i=1,...,m$, and $T_\ell = dL_g(\ve_\ell)$, $\ell=1,...,k$, are the left-invariant vector fields generated by such basis, we assume that $\bG$ is endowed with a left-invariant Riemannian
tensor with respect to which the vector fields $X_1,\ldots,X_m$,
$T_1,\ldots,T_k$ are orthonormal at every point. 
The sub-Laplacian with respect to the basis $\{e_1,...,e_m\}$ is given by $\Delta_H = \sum_{i=1}^m X_i^2.$ Consider the analytic mappings $z:\bG\to V_1$, $t:\bG\to V_2$ uniquely defined through the equation $g = \exp(z(g)+t(g))$. For each $i=1,...,m$ we set
\[
z_i = z_i(g) = <z(g),e_i>,
\]
whereas for $s=1,...,k$ we let
\[
t_s = t_s(g) = <t(g),\ve_s>.
\]
We will indicate with $(z,t)\in \bg$ the exponential coordinates of a point $g\in \bG$.

Consider the linear mapping $J:V_2 \to$ End$(V_1)$, defined by 
\begin{equation}\label{J}
<J(t)z,z'> = <[z,z'],t>.
\end{equation}

\begin{dfn}\label{D:Htype}
A Carnot group of step two is called of \emph{Heisenberg type} if for every $t\in V_2$ such that $|t| = 1$, the mapping $J(t)$ is orthogonal. This is equivalent to saying that for every $z, z'\in V_1$, and every $t\in V_2$, one has
\begin{equation}\label{ht}
<J(t)z,J(t)z'> = |t|^2 <z,z'>.
\end{equation}
\end{dfn}
In particular, when $\bG$ is of Heisenberg type, then \eqref{ht} with $z=z'$ yields
\begin{equation}\label{ort}
|J(t)z|^2 = <J(t)z,J(t)z> = |t|^2 |z|^2.
\end{equation}

We will  need the following simple, yet crucial, result about groups of Heisenberg type. It follows from \eqref{ort} by polarization.

\begin{lemma}\label{L:ip}
If $\bG$ is of Heisenberg type, then for every $t, t'\in V_2$ and any $z \in V_1$ one has
\[
<J(t)z,J(t')z> = |z|^2 <t,t'>.
\]
In particular, we obtain for every $\ell, \ell' = 1,...,k$
\[
<J(\ve_\ell)z,J(\ve_{\ell'})z> = |z|^2 \delta_{\ell \ell'}.
\] 
\end{lemma}

\begin{dfn}\label{D:Metivier}
A Carnot group of step two is called a \emph{Metivier group} if there exists a constant $B>0$ such that
\[
|J(t)z| \ge B |z| |t|,\ \ \ \ \ \ z\in V_1,\ t\in V_2.
\]
\end{dfn}

\begin{rmrk}\label{R:metivier}
Clearly, every group of Heisenberg type is a Metivier group. The opposite inclusion is false. For instance (see Remark 3.7.5 in \cite{BLU}), consider $\bG = \R^5 = \R^4_z \times \R_t$ with the group law
\[
(z,t) (z',t') = (z+z',t+t'+\frac{1}{2}<Az,z'>),
\]
where $A$ is the $4\times 4$ skew-symmetric matrix
\[
A = \begin{pmatrix}
0 & - 1 & 0 & 0
\\
1 & 0 & 0 & 0
\\
0 & 0&  0 & - 2
\\
0 & 0 & 2 & 0
\end{pmatrix}.
\]
Since \emph{det}$\ A\not= 0$, the matrix $A$ is non-singular and therefore $\bG$ is a Metivier group. But $\bG$ is not of $H$-type since $A\not\in O(4)$. 
\end{rmrk}

The following expressions of the vector fields $X_i$, and of the sub-Laplacian $\Delta_H$ in exponential coordinates will be useful, see \cite{GV}.

\begin{lemma}\label{L:Xec}
Let $\bG$ be a Carnot group of step two. Then, in the exponential coordinates $(z,t)$ one has
\begin{equation}\label{Xi}
X_i  = \p_{z_i} + \frac 12 \sum_{\ell=1}^k <J(\ve_\ell)z,e_i> \partial_{t_\ell},
\end{equation}
\begin{equation}\label{sl}
\Delta_H = \Delta_z + \frac 14 \sum_{\ell,\ell' = 1}^k <J(\ve_\ell)z,J(\ve_{\ell'})z> \p_{t_\ell}\p_{t_{\ell'}} + \sum_{\ell = 1}^k \p_{t_\ell} \Theta_\ell,
\end{equation}
where $\Delta_z$ represents the standard Laplacian in the variable $z = (z_1,...,z_m)$, and
\begin{equation}\label{thetaell}
\Theta_\ell = \sum_{i=1}^m <J(\ve_\ell)z,e_i> \p_{z_i}.
\end{equation}
In particular, when $\bG$ is of Heisenberg type we obtain
\begin{equation}\label{slH}
\Delta_H = \Delta_z + \frac{|z|^2}{4} \Delta_t  + \sum_{\ell = 1}^k \p_{t_\ell} \Theta_\ell.
\end{equation}
\end{lemma}

We will also need the following.

\begin{lemma}\label{L:gauge}
Let $\bG$ be a Carnot group of step two, and consider the gauge $\rho= (|z|^4 + 16 |t|^2)^{1/4}$. Then,
\begin{equation}\label{gradgauge}
\rh = \frac{1}{\rho^6} \left(|z|^6 + 16 |J(t)z|^2\right).
\end{equation}
If $\bG$ is of Heisenberg type, then
\begin{equation}\label{gradgaugeH}
\rh = \frac{|z|^2}{\rho^2}.
\end{equation}
\end{lemma}

\begin{proof}
In what follows we let $r = r(g) = |z|, s = s(g) = |t|$, so that
\[
\rho = (r^4 + 16 s^2)^{1/4}.
\]
Using Lemma \ref{L:Xec} we find
\begin{align}\label{XiN}
X_i \rho & = \frac{1}{r} \frac{\partial \rho}{\partial r} z_i + \frac{1}{2 s} \frac{\partial \rho}{\partial s} \sum_{\ell=1}^k <J(\ve_\ell)z,e_i> t_\ell
\\
& = \frac{1}{r} \frac{\partial \rho}{\partial r} z_i + \frac{1}{2 s} \frac{\partial \rho}{\partial s} <J(t)z,e_i>.
\notag
\end{align}
Using the fact that $<J(t)z,z> = 0$ for every $t\in V_2, z\in V_1$, this gives
\begin{align*}
\rh & = \left(\frac{\partial \rho}{\partial r}\right)^2 + \frac{1}{4s^2}\left(\frac{\partial \rho}{\partial s}\right)^2 \sum_{i=1}^m <J(t)z,e_i>^2
\\
& = \left(\frac{\partial \rho}{\partial r}\right)^2 + \frac{1}{4s^2}\left(\frac{\partial \rho}{\partial s}\right)^2 |J(t)z|^2.
\end{align*}
Now we observe that
\begin{equation}\label{rd}
\frac{\p \rho}{\p r} = \frac{r^3}{\rho^3},\ \ \ \ \ \frac{\partial \rho}{\partial s} =  \frac{8s}{\rho^3}.
\end{equation}
Inserting these formulas in the last equation we obtain \eqref{gradgauge}. When $\bG$ is of Heisenberg type using \eqref{ort} in \eqref{gradgauge}, we obtain \eqref{gradgaugeH}.

\end{proof}

We next recall a beautiful result which is due to Folland for the Heisenberg group $\Hn$, see \cite{F1}, and to Kaplan for groups of Heisenberg type, see \cite{K}. 

\begin{prop}\label{P:FK}
Let $\bG$ be a group of Heisenberg type, and define the constant $C>0$ by the formula
\[
C^{-1} = m(Q-2) \int_{\bG} \frac{dz dt}{\left[(|z|^2+1)^2 + 16 |t|^2\right]^{\frac{Q+2}{4}}}.
\]
Then, the function 
\[
\Gamma(z,t) = \frac{C}{\rho(z,t)^{Q-2}},
\]
is a fundamental solution with singularity at $e$ for the sub-Laplacian associated with the orthonormal basis $\{e_1,...,e_m\}$.
\end{prop}

As a consequence of Proposition \ref{P:FK} we see that, up to a renormalization, the sets $B_r$ defined in \eqref{balls} are now given by $B_r = \{(z,t)\in \bG\mid \rho(z,t)<r\}$, where
\begin{equation}\label{ga}
\rho= (|z|^4 + 16 |t|^2)^{1/4}.
\end{equation}

The next lemma provides an expression for the discrepancy of a function $u$ at $e\in \bG$ in any group of Heisenberg type.

\begin{lemma}\label{L:innerp}
Let $\bG$ be a group of Heisenberg type and let $\rho$ be as in \eqref{ga}. Given a function $u:\bG\to \mathbb R$,  its discrepancy at $e$ is given by
\begin{align*}
E_u = <\nabla_H u,\nabla_H \rho> - \frac{Zu}{\rho} |\nabla_H \rho|^2  = \frac{4}{\rho^3} \sum_{\ell=1}^k t_\ell \Theta_\ell(u),
\end{align*}
where $\Theta_\ell$ is the vector field defined in \eqref{thetaell} above. As a consequence, in a group of Heisenberg type a function $u$ has vanishing discrepancy if and only if
\[
\sum_{\ell=1}^k t_\ell \Theta_\ell(u) \equiv 0.
\]
\end{lemma}

\begin{proof}
The equations \eqref{Xi} and \eqref{XiN} give
\begin{align*}
& <\nabla_H u, \nabla_H \rho> = \sum_{i=1}^m X_i u X_i \rho
\notag\\
& = \sum_{i=1}^m \left(\frac{\partial u}{\partial z_i} + \frac 12 \sum_{\ell=1}^k <J(\ve_\ell)z,e_i> \frac{\partial u}{\partial t_\ell}\right) \left(\frac{1}{r} \frac{\partial \rho}{\partial r} z_i + \frac{1}{2 s} \frac{\partial \rho}{\partial s} <J(t)z,e_i> \right)
\notag\\
& = \frac{1}{r}\ \frac{\partial \rho}{\partial r} \sum_{i=1}^m z_i \frac{\partial u}{\partial z_i}  + \frac{1}{2 s} \frac{\partial \rho}{\partial s} \sum_{i=1}^m \frac{\partial u}{\partial z_i} <J(t)z,e_i> 
\notag\\
& + \frac{1}{2 r} \frac{\partial \rho}{\partial r} \sum_{\ell=1}^k \sum_{i=1}^m z_i <J(\ve_\ell)z,e_i> \frac{\partial u}{\partial t_\ell}
\notag\\
& + \frac{1}{4 s} \frac{\partial \rho}{\partial s} \sum_{\ell=1}^k \left(\sum_{i=1}^m <J(\ve_\ell)z, e_i><J(t)z, e_i>\right) \frac{\partial u}{\partial t_\ell}
\\
& = (I) + (II) + (III) + (IV).
\end{align*}
We now have
\[
(III) = \frac{1}{2 r} \frac{\partial \rho}{\partial r} \sum_{\ell=1}^k \sum_{i=1}^m z_i <J(\ve_\ell)z,e_i> \frac{\partial u}{\partial t_\ell} = \frac{1}{2 r} \frac{\partial \rho}{\partial r} \sum_{\ell=1}^k  <J(\ve_\ell)z,z> \frac{\partial u}{\partial t_\ell} = 0.
\]
 On the other hand, we obtain from Lemma \ref{L:ip}
\begin{align*}
(IV) & = \frac{1}{4 s} \frac{\partial \rho}{\partial s} \sum_{\ell =1}^k \left(\sum_{i=1}^m <J(\ve_\ell)z, e_i><J(t)z,e_i>\right) \frac{\partial u}{\partial t_\ell}
\\
& = \frac{1}{4 s} \frac{\partial \rho}{\partial s} \sum_{\ell =1}^k <J(t)z,J(\ve_\ell)z> \frac{\partial u}{\partial t_\ell}.
\end{align*}

\[
= \frac{r^2}{4 s} \frac{\partial \rho}{\partial s} \sum_{\ell=1}^k t_{\ell} \frac{\partial u}{\partial t_\ell} = \frac{2}{\rho} |\nabla_H \rho|^2  \sum_{\ell=1}^k  t_{\ell} \frac{\partial u}{\partial t_\ell},
\]
where we have used \eqref{rd} and \eqref{gradgaugeH}.
We thus conclude that
\[
(I) + (III) + (IV) = \frac{Zu}{\rho} \rh.
\] 
Finally, it is clear from \eqref{thetaell} above that
\[
(II) = \frac{4}{\rho} \sum_{\ell=1}^k \frac{t_\ell}{\rho^2} \Theta_\ell(u).
\]
This completes the proof.

\end{proof}

In view of Lemma \ref{L:innerp} in a group of Heisenberg type we have the following complete characterization of functions with vanishing discrepancy. 

\begin{prop}\label{P:cvd}
Let $\bG$ be a group of Heisenberg type. A harmonic function $u$ in $B_R$ has vanishing discrepancy at $e$ if and only if $u$ solves the Baouendi operator
\[
\mathcal B_1 u = \Delta_z u + \frac{|z|^2}{4} \Delta_t u = 0,\ \ \ \ \ \ \text{in}\ B_R.
\]
\end{prop}

\begin{proof}
Immediate consequence of Lemma \ref{L:innerp} and of \eqref{slH}.

\end{proof} 

Two sufficient, but by no means necessary, conditions for vanishing discrepancy are expressed by the following result. We recall that in a group of Heisenberg type (or even, more in general, in a Metivier group) the complex structure induced by the map $J$ forces the first layer $V_1$ of the Lie algebra to be even dimensional, i.e., $m = 2 n$ for some $n\in \mathbb N$. Let us write $z \in V_1$ as $z = (x,y)\cong \R^{2n}$, and indicate with $w_i = (x_i,y_i)\in \R^2, i=1,...,n$, so that
 \[
r_i = |w_i| = \sqrt{x_i^2 + y_i^2},\ \ \ \ \ \ \ i=1,...,n.
\]

\begin{dfn}\label{D:cylpoly}Let $\bG$ be a group of Heisenberg type and $u:\bG\rightarrow \R$.
 \begin{itemize}
  \item[(a)] We say that $u$ has \emph{cylindrical symmetry} if in the exponential coordinates one has
  \begin{equation*}
   u(g)=\vf(|z|,t)
  \end{equation*}
  for some function $\vf:[0,\infty)\times V_2\rightarrow \R$.

  \item[(b)] We say that $u$ has \emph{polyradial symmetry} if instead one has
  \begin{equation*}
   u(g)=\vf(|w_1|,\ldots,|w_n|,t)
  \end{equation*}
  for some function $\vf:[0,\infty)^n \times V_2\rightarrow \R$.
 \end{itemize}
\end{dfn}

\begin{prop}\label{P:innerp2}
Let $\bG$ be a group of Heisenberg type. 
\begin{enumerate}
 \item[(a)] If $u$ has cylindrical symmetry, then 
 \[
<\nh u,\nh \rho> = \frac{Zu}{\rho} \rh,
\]
and therefore $E_u \equiv 0$.

\item[(b)] In the case of the Heisenberg group $\Hn$, if we make the weaker assumption that $u$ is a polyradial function, then $E_u\equiv 0$.
\end{enumerate}
\end{prop}

\begin{proof}
If $u(g) = \vf(r,t)$, with $r = |z|$, we have
\[
\partial_{z_i} u = \frac{1}{r} \frac{\partial \vf}{\partial r} z_i, \quad\quad\quad\quad i = 1,...,m,
\]
which gives
\[
\Theta_\ell u = \sum_{i=1}^m <J(\ve_\ell)z,e_i> \p_{z_i} u = \frac{1}{r} \frac{\partial \vf}{\partial r} <J(\ve_\ell)z,z> = 0,\ \ \ \ \ \ell = 1,...,k.
\]
The conclusion of (a) follows from Lemma \ref{L:innerp}. 

If $\bG=\Hn$ for some $n$, there is only one $\Theta_\ell$, given in equation \eqref{slHn} as
\begin{equation}
 \label{E:ThetaH}\Theta=\sum_{j=1}^n (x_j\partial_{y_j}-y_j\partial_{x_j}).
\end{equation}
If in this case we assume that $u(g) = \vf(|w_1|,...,|w_n|,t)$, then for $j=1,2,\ldots,n$,
\begin{align}
\label{E:polyderivx}\partial_{x_j} u &= \frac{\p \vf}{\p r_j} \frac{x_j}{r_j}\\
\label{E:polyderivy}\partial_{y_j} u &= \frac{\p \vf}{\p r_j} \frac{y_j}{r_j}.
\end{align}
Inserting the information from \eqref{E:polyderivx} and \eqref{E:polyderivy} into \eqref{E:ThetaH}, we have
\begin{align*}
 \Theta u&=\sum_{j=1}^n x_jy_j\frac{\p \vf}{\p r_j}-y_jx_j\frac{\p \vf}{\p r_j}=0.
\end{align*}
Again, the result follows from Lemma \ref{L:innerp}.
\end{proof}

It is somewhat surprising that for general groups of Heisenberg type, it is not necessarily true that polyradial functions have zero discrepancy. This is illustrated by the following example, which is itself inspired by an example found in chapter 18 of \cite{BLU}.

Let $V_1=\R^4$ and $V_2=\R^2$. Choose a basis $\{e_1,e_2,e_3,e_4\}$ of $V_1$, and $\{\ve_1,\ve_2\}\in V_2$, and let $<\cdot,\cdot>$ be an inner product on $\mathfrak{g}=V_1\oplus V_2$ which makes the set $\{e_1,e_2,e_3,e_4,\ve_1,\ve_2\}$ orthonormal. We create a Lie algebra on $\mathfrak{g}$ by requiring that
 \begin{align*}
  [e_1,e_2]&=-[e_2,e_1]=[e_3,e_4]=-[e_4,e_3]=\ve_1\\
  -[e_1,e_3]&=[e_3,e_1]=[e_2,e_4]=-[e_4,e_2]=\ve_2.
 \end{align*}
 and all other brackets are equal to zero. Since $[\mathfrak{g},\mathfrak{g}]=V_2$ and $[V_2,\mathfrak{g}]=\{0\}$, the bracket of any three vectors in $\mathfrak{g}$ is zero. Hence the Jacobi identity is trivially satisfied and the resulting Lie algebra is nilpotent of step 2. By the results of section 2.2 in \cite{BLU}, this generates a Lie group $\bG$ with Lie algebra isomorphic to $\mathfrak{g}$. By the very definition of the bracket, $[V_1,V_1]=V_2$ and $[V_1,V_2]=[V_2,V_2]=\{0\}$, so $\bG$ is in fact a step 2 Carnot group with topological dimension $n=6$ and homogeneous dimension $Q=4+2\cdot 2=8$. Now,
 \begin{align*}
  <J(\ve_1)e_1,e_k>&= \delta_{2k}&&<J(\ve_1)e_2,e_k>=-\delta_{1k},
  \\
  <J(\ve_1)e_3,e_k>&= \delta_{4k}&&<J(\ve_1)e_4,e_k>=-\delta_{3k},
  \\
  <J(\ve_2)e_1,e_k>&= -\delta_{3k}&&<J(\ve_2)e_2,e_k>=\delta_{4k},
  \\
  <J(\ve_2)e_3,e_k>&= \delta_{1k}&&<J(\ve_2)e_4,e_k>=-\delta_{2k}.
 \end{align*}
 
 Therefore, given  $t=t_1\ve_1+t_2\ve_2$, the matrix for $J(t)$ in the basis $\{e_1,e_2,e_3,e_4\}$ is given by
 \begin{equation*}
  \left(\begin{matrix}0&-t_1&t_2&0\\t_1&0&0&-t_2\\-t_2&0&0&-t_1\\0&t_2&t_1&0\end{matrix}\right).
 \end{equation*}
This is clearly orthogonal if $|t|=1$, hence $\bG$ is of Heisenberg type. 
Let us compute the vector fields $\Theta_\ell$, $\ell=1,2$ for the group $\bG$. For \[
 z=x_1e_1+x_2e_2+x_3e_3+x_4e_4\in \mathfrak{g},
 \]
  one has
 \begin{equation*}
  \left(\begin{matrix}0&-1&0&0\\1&0&0&0\\0&0&0&-1\\0&0&1&0\end{matrix}\right)\left(\begin{matrix}x_1\\x_2\\x_3\\x_4\end{matrix}\right)=\left(\begin{matrix}-x_2\\x_1\\-x_4\\x_3\end{matrix}\right)\quad\quad \left(\begin{matrix}0&0&1&0\\0&0&0&-1\\-1&0&0&0\\0&1&0&0\end{matrix}\right)\left(\begin{matrix}x_1\\x_2\\x_3\\x_4\end{matrix}\right)=\left(\begin{matrix}x_3\\-x_4\\-x_1\\x_2\end{matrix}\right).
 \end{equation*}
 so that, by Lemma \ref{L:Xec},
 \begin{align}
 \Theta_1&=-x_2\partial_{x_1}+x_1\partial_{x_2}-x_4\partial_{x_3}+x_3\partial_{x_4}\label{E:Tpc1},\\
 \Theta_2&=x_3\partial_{x_1}-x_4\partial_{x_2}-x_1\partial_{x_3}+x_2\partial_{x_4}\label{E:Tpc2}.
 \end{align}

 Consider the function $u:\bG\rightarrow \R$ given in exponential coordinates by
 \begin{equation*}
  u(z,t)=x_1^2+x_3^2=|w_1|^2.
 \end{equation*}
 By (b) in Definition \ref{D:cylpoly}, $u$ is a polyradial function (we recall that $x=(x_1,x_2)$ and $y=(x_3,x_4)$ in this setting). Using \eqref{E:Tpc1} and \eqref{E:Tpc2}, we have
 \begin{equation*}
  t_1\Theta_1 u+t_2\Theta_2=-2t_1(x_1x_2+x_3x_4)\not\equiv 0.
 \end{equation*}
That the function $u$ does not have vanishing discrepancy now follows from Lemma \ref{L:innerp}.

It is interesting to note that $\Theta_2u=0$ for the function $u$ above. This might suggest that by pairing the variables in a different fashion may lead to a polyradial-type function with vanishing discrepancy. This is not the case. If one defines
\begin{equation*}
u_{ij}(z,t)=x_i^2+x_j^2\quad\quad i,j=1,2,3,4,\ i\ne j,
 \end{equation*}
then it is easy to check that $t_1\Theta_1+t_2\Theta_2$ does not annihilate any of these functions.

\section{Monotonicity formulas for the Baouendi operators $\Ba$}\label{S:bg}

In this final section we turn our attention to the \emph{Baouendi operators} 
\begin{equation}\label{bg}
 \Ba=\Delta_z+\frac{|z|^{2\alpha}}{4}\Delta_t,\ \ \ \ \ \ \alpha >0.
\end{equation}
As we have seen in Section \ref{S:examples}, when $\alpha = 1$ these operators are intimately connected to the sub-Laplacians in groups of Heisenberg type. We recall that in the paper \cite{G} the first named author had introduced a frequency function associated with $\Ba$, and proved that such frequency is monotone nondecreasing on solutions of $\Ba u = 0$, see Theorem \ref{T:almgrenBa} below. A version of this Almgren type monotonicity formula for $\Ba$ played an extensive role also in the mentioned recent work \cite{CSS} on the Signorini problem. In this section we establish some new monotonicity properties of the operators $\Ba$, and use them to extract some new information about the local and global nature of solutions of $\Ba u = 0$.
 
We  let $m$ and $k$ be fixed positive integers, $z\in \R^m$, $t\in \R^k$, and we denote $N=m+k$. We observe that $\Ba$ is not translation invariant in $\R^N$. However, it is invariant with respect to the translations along the $k$-dimensional subspace $\mathcal M = \{0\}\times \R^k$. Also, $\Ba$ can be written in two useful ways. As a divergence form operator, $\Ba u=\operatorname{div}(A\nabla u)$, where $A=A_\alpha$ is the $N\times N$ block matrix given by
\begin{align*}
 A=\left(\begin{matrix}
        I_m & 0\\
        0 & \frac{|z|^{2\alpha}}{4}I_k
       \end{matrix}
\right).
\end{align*}
The operator can also be written as a sum-of-squares of not necessarily smooth vector fields. If we in fact denote
\begin{align}\label{XBa}
 X_j=\begin{cases}
      \partial_{z_j}&j=1,2,\ldots,m\vspace{0.2cm}\\
      \displaystyle\frac{|z|^\alpha}{2}\partial_{t_{j-m}}&j=m+1,m+2,\ldots,m+k=N,
     \end{cases}
\end{align}
then clearly
\begin{align}
 \Ba=\sum_{j=1}^N X_j^2.
\end{align}
It thus follows by H\"{o}rmander's theorem that $\Ba$ is hypoelliptic if $\alpha=2\ell$ is an even integer. 
We equip $\R^N$ with the following non-isotropic dilations
\begin{align}\label{dilB}
 \delta_\lambda(z,t)&=(\lambda z,\lambda^{\alpha+1}t),\ \ \lambda>0.
\end{align}
We say that a function $u$ is $\delta_\la$-homogeneous (or simply, homogeneous) of degree $\kappa$ if 
\[
u(\delta_\la(z,t)) = \la^\kappa u(z,t),\ \ \ \ \ \la>0.
\]
It is straightforward to verify that $\Ba$ is $\delta_\la$-homogeneous of degree two, i.e., 
\begin{equation*}
 \Ba(\delta_\lambda \circ u)=\lambda^2\delta_\lambda\circ (\Ba u).
\end{equation*}

Functions which satisfy $\Ba u=0$ may have a certain degree of singularity along the manifold $\mathcal M = \{0\}\times \R^k$. It is thus necessary to introduce the following classes of ``smooth  functions''
\begin{align*}
 \Gamma^1_\alpha(\Omega)=\{f\in C(\Omega)\mid f,X_jf\in C(\Omega)\},
\end{align*}
where the derivatives taken are weak derivatives. We also set
\[
\Gamma^2_\alpha(\Omega)=\{f\in C(\Omega)\mid f,X_jf\in \Gamma^1_\alpha(\Omega)\}
\]
Thus, solutions to $\Ba u=0$ in $\Omega$ are taken to be of class $\Gamma^2_\alpha(\Omega)$.

The infinitesimal generator of the dilations \eqref{dilB} is given by
\begin{equation}\label{Za}
 Z_\alpha=\sum_{i=1}^m z_i\partial_{z_i} + (\alpha+1)\sum_{j=1}^k t_j\partial_{t_j}.
\end{equation}
It is easy to verify that $u\in \Gamma^1_\alpha$ is homogeneous of degree $\kappa$ if and only if
\begin{equation}\label{Zak}
\Za u = \kappa u.
\end{equation}
We note that 
\begin{align*}
 d(\delta_\lambda(z,t))&=\lambda^{m+(\alpha+1)k}\ dz\ dt,
\end{align*}
which motivates the definition of the \emph{homogeneous dimension} for the number
\begin{align}\label{Qa}
 Q=Q_\alpha=m+(\alpha+1)k.
\end{align}

In the study of the operators \eqref{bg} the following pseudo-gauge on $\R^{N}$ plays an interesting role:
\begin{align}\label{ra}
 \rho_\alpha(z,t) & = \left(|z|^{2(\alpha+1)}+4(\alpha+1)^2|t|^2\right)^{\frac1{2(\alpha+1)}}.
\end{align}
We emphasize that this is not a true gauge, because there is no underlying group structure associated to $\Ba$. Accordingly, the ball and sphere centered at the origin with radius $r>0$ are respectively defined as
\begin{align}\label{BSr}
 B_r&=\{(z,t)\in \R^N \mid \rho(z,t)<r\},\ \ \ \ \ S_r = \p B_r.
\end{align}
In \cite{G}, the first named author proved that, with $C_\alpha>0$ given by
\[
C_\alpha^{-1} = (m+\alpha-1)(Q-2) \int_{\R^N} \frac{|z|^{\alpha -1} dz dt}{\left[(|z|^{\alpha+1} + 1)^2 + 4(\alpha +1)^2 |t|^2\right]^{\frac{Q+2\alpha}{2(\alpha+1)}}},
\]
 the function 
\begin{equation}\label{Ga}
 \Gamma(z,t)=\frac{C_{\alpha}}{\ra(z,t)^{Q-2}}
\end{equation}
is a fundamental solution for $-\Ba$ with singularity at $(0,0)$. Since, as we observed above, the operator is invariant with respect to translations along $\mathcal M = \{0\}\times \R^k$, this gives a fundamental solution for this entire subspace of $\R^N$.

For $u, v\in \Gamma^1_\alpha(\R^{N})$, we define the $\alpha$-gradient of $u$ to be
\begin{equation*}
 \nabla_\alpha u=\nabla_z u+\frac{|z|^\alpha}{2}\nabla_t u,
\end{equation*}
and we set
\[
<\na u,\na v> = <\nabla_z u,\nabla_z v> + \frac{|z|^{2\alpha}}{4} <\nabla_t u,\nabla_t v>.
\]
The square of the length of $\nabla_\alpha u$ is
\begin{equation*}
 |\nabla_\alpha u|^2=|\nabla_z u|^2+\frac{|z|^{2\alpha}}{4}|\nabla_t u|^2.
\end{equation*}
The following lemma, collects the identities (2.12)-(2.14) in \cite{G}.

\begin{lemma}\label{l:gradalpharho}
 Given a function $u$ one has in $\R^N\setminus\{0\}$,
 \begin{equation}\label{nara}
  \psi_{\alpha} \overset{def}{=} |\na \ra|^2 = \frac{|z|^{2\alpha}}{\ra^{2\alpha}},
 \end{equation}
and
 \begin{equation}\label{nara2}
 <\na u,\na \ra> = \frac{\Za u}{\ra} \psi_\alpha.
 \end{equation}
 \end{lemma}

 Given a function $u\in \Gamma^1_\alpha(B_R)$, we now define, in analogy to what was done in Definition \ref{D:DH} for Carnot groups, the Dirichlet integral, height, and frequency of $u$ in $B_R$, respectively:
\begin{align}
 D(r)&= D(u,r) = \int_{B_r}|\nabla_\alpha u|^2 dzdt,\ \ \ 0<r<R, \label{E:defD}\\
 H(r)&= H(u,r) = \int_{S_r} u^2\frac{\psi_\alpha}{|\nabla \ra|}\ dH_{N-1}\\
 N(r)& = N(u,r) =\frac{rD(r)}{H(r)}.
\end{align}

The following proposition is formula (4.36) in \cite{G}, and it represents the counterpart to Proposition \ref{P:beauty}.
\begin{prop}\label{P:diri}
 Let $u$ be a solution of $\Ba u=0$ in $B_R$. Then for every $0<r<R$,
 \begin{align*}
  D(r)&=\int_{S_r}u\left(\frac{Z_\alpha u}{r}\right)\frac{\psi_\alpha}{|\nabla \ra|}dH_{N-1}.
 \end{align*}
\end{prop}

The next lemma is the analogue of Lemma \ref{L:H'} for the operators $\Ba$.

\begin{lemma}\label{l:hnz}
 Let $u$ be a solution of $\Ba u=0$ in $B_R$. Then, either $u\equiv 0$ in $B_{R}$, or $H(r)\ne 0$ for every $r>0$.
\end{lemma}

We now recall two lemmas which will be important in the sequel. For their proofs we refer the reader to  \cite{G}. 

\begin{lemma}[see Lemma 4.1 in \cite{G}]\label{l:hprime}
 Let $u$ be a solution of $\Ba u=0$ in $B_R$. Then
 \[
  H'(r)=\frac{Q-1}{r}H(r)+2D(r),\hspace{1cm}\textrm{ for }r\in (0,R).
 \]
\end{lemma}

By using the coarea formula, we have the following for the derivative of $D$:
\begin{equation}
 D'(r)=\int_{S_r} |\nabla_\alpha u|^2\frac{dH_{N-1}}{|\nabla\ra|}.
\end{equation}

The next result, which is Corollary 2.3 in \cite{G}, constitutes a remarkable property of solutions of $\Ba u = 0$.

\begin{lemma}[First variation of the energy]\label{l:dprime}
 Let $u$ be a solution of $\Ba u=0$ in $B_R$. Then,
 \begin{equation}
  D'(r)=2\int_{S_r}\left(\frac{Z_\alpha u}{r}\right)^2\frac{\psi_\alpha}{|\nabla \ra|}dH_{N-1}+\frac{Q-2}{r}D(r).
 \end{equation}

\end{lemma}

The following result represents a generalization of the Almgren's monotonicity formula to solutions of the operator $\Ba$.

\begin{thrm}[Almgren type monotonicity formula, see \cite{G}]\label{T:almgrenBa}
Let $u$ be a solution of $\Ba u = 0$ in $B_R$, and suppose that for no $r\in (0,R)$ we have $u\equiv 0$ in $B_r$. Then, the function $r\to N(u,r)$ is nondecreasing on $(0,R)$.
\end{thrm}

\begin{proof}
Using Lemmas \ref{l:hprime}, \ref{l:dprime} and Proposition , we find
\begin{equation}\label{logN}
\frac{d}{dr} \log N(r) = 2 \frac{\int_{S_r}\left(\frac{Z_\alpha u}{r}\right)^2\frac{\psi_\alpha}{|\nabla \ra|}dH_{N-1}}{\int_{S_r}u\left(\frac{Z_\alpha u}{r}\right)\frac{\psi_\alpha}{|\nabla \ra|}dH_{N-1}} - 2 \frac{\int_{S_r}u\left(\frac{Z_\alpha u}{r}\right)\frac{\psi_\alpha}{|\nabla \ra|}dH_{N-1}}{\int_{S_r} u^2\frac{\psi_\alpha}{|\nabla \ra|}dH_{N-1}}\ge 0,
\end{equation}
where we have used Cauchy-Schwarz inequality.

\end{proof}

Proposition \ref{P:diri} has the following consequence.

\begin{cor}\label{C:homo}
 Let $u$ be a solution of $\Ba u=0$ in $B_R$. Then, $N(u,r) \equiv \kappa$, $0<r<R$, if and only if $u$ is homogeneous of degree $\kappa$ in $B_R$.
 \end{cor}
 
 \begin{proof}
Suppose that $u$ is homogeneous of degree $\kappa$ in $B_R$. By the hypothesis and \eqref{Zak} we have $\Za u = \kappa u$. Proposition \ref{P:diri} thus gives $D(r) = \frac{\kappa}{r} H(r)$, and therefore $N(r) \equiv \kappa$ in $(0,R)$.
The opposite implication can be proved from \eqref{logN} by arguing exactly as in the proof of Proposition \ref{P:monhom} above, and we omit the relevant details.
\end{proof}

We have seen in Corollary \ref{C:homo} that the frequency of a homogeneous solution of $\Ba u = 0$ is constant. It is natural to ask whether the opposite result hold, i.e., whether it is true that if a solution of $\Ba u = 0$ has constant frequency, then $u$ must be a homogeneous harmonic function. In order to answer this question, we now   establish, for solutions of the operator $\Ba$, a monotonicity  result similar to Theorem \ref{T:weissG}.

\begin{thrm}\label{T:weissB}
 Let $u$ be a solution of $\Ba u=0$ in $B_R$. For every $\kappa\ge 0$ define for $0<r<R$
 \begin{equation}\label{WBa}
  \mathcal{W}_\kappa(u,r)=\frac{1}{r^{Q-2+2\kappa}}D(u,r)-\frac{\kappa}{r^{Q-1+2\kappa}}H(u,r).
 \end{equation}
 Then,
 \begin{equation}\label{W'Ba}
  \frac{d}{dr}\mathcal{W}_\kappa(u,r)=\frac{2}{r^{Q+2\kappa}}\int_{S_r}(Z_\alpha u-\kappa u)^2\frac{\psi_\alpha}{|\nabla \ra|}dH_{N-1}.
 \end{equation}
 Consequently, $r\mapsto \mathcal{W}_\kappa(u,r)$ is a non-decreasing function in $(0,R)$, and $\mathcal{W}_\kappa(u,\cdot)$ is constant if and only if $u$ is homogeneous of degree $\kappa$.
\end{thrm}

The proof of Theorem \ref{T:weissB}, based on Lemmas \ref{l:hprime} and \ref{l:dprime}, is nearly identical to that of Theorem \ref{T:weissG} and therefore we omit it. 

\begin{rmrk}\label{R:sh}
We note that, similarly to the case of groups (see Remark \ref{R:hom} above), the sufficiency part of Corollary \ref{C:homo} also follows from Theorem \ref{T:weissB}.
Similarly to \eqref{Wk}, we  write \eqref{WBa} as follows
\begin{equation}\label{WkBa}
\mathcal W_\kappa(u,r) = \frac{H(u,r)}{r^{Q-1+2\kappa}} \big(N(u,r) - \kappa\big).
\end{equation}
By the hypothesis that $N(u,r) = \kappa$ for $0<r<R$, we see that $\mathcal W_\kappa(u,\cdot) =0$ on $(0,R)$, and therefore $d/dr \mathcal W_\kappa(u,\cdot) = 0$. In view of \eqref{W'Ba} this gives $\Za u = \kappa u$ in $B_R$.
\end{rmrk}

\begin{dfn}\label{D:shBa}
Let $\kappa\ge 0$. We denote by $\mathfrak P_{\alpha,\kappa}(\R^N)$ the space of all functions $P_\kappa\in \Gamma^2_\alpha(\R^N)$ such that $\Ba P_\kappa = 0$ and $\Za P_\kappa = \kappa P_\kappa$. The elements of such space will be called $\Ba$-\emph{solid harmonics} of degree $\kappa$. 
\end{dfn} 

We emphasize that, for $P_\kappa\in \mathfrak P_{\alpha,\kappa}(\R^N)$, the number $\kappa$ needs not be an integer. For instance, if $A = \frac{(\alpha+1)(2\alpha+m)}{k}$, then the function $P_\kappa(z,t) =  |z|^{2(\alpha + 1)} - A |t|^2,$ is a solution of $\Ba f = 0$, homogeneous of degree $\kappa = 2(\alpha+1)$. Thus, $P_\kappa \in \mathfrak P_{\alpha,\kappa}(\R^N)$.

\begin{prop}\label{P:cm}
For every $\kappa\ge 0$ the space $\mathfrak P_{\alpha,\kappa}(\R^N)$ is finite dimensional.
\end{prop}

\begin{proof}
Let $u\in  \mathfrak P_{\alpha,\kappa}(\R^N)$. Since $\Za u = \kappa u$ in $\R^N$,
for any $(z,t)\in \R^N$ such that $\ra(z,t)\ge 1$ we must have
\[
u(z,t) = \ra(z,t)^\kappa u(\delta_{\ra(z,t)^{-1}}(z,t)),
\]
and therefore
\[
|u(z,t)| \le \left(\underset{S_1}{\max}\ |u|\right)  \ra(z,t)^\kappa.
\]
This estimate gives 
\begin{equation}\label{growth}
\underset{r\ge 1}{\sup}\ \left(\frac{1}{r^\kappa} \underset{B_r}{\sup} |u|\right) <\infty.
\end{equation}
With \eqref{growth} in hands, we can now invoke the Colding-Minicozzi type theorem at the end of the paper \cite{KL} by Kogoj and Lanconelli to conclude that $\mathfrak P_{\alpha,\kappa}(\R^N)$ is finite dimensional.

\end{proof}  

It is quite notable that solid harmonics of different degrees enjoy the following orthogonality property. It is well-known that a similar orthogonality property fails  for the solid harmonics in the Heisenberg group $\Hn$.

\begin{prop}\label{P:ortho}
For every $\kappa\not= \mu$, let $P_\kappa\in \mathfrak P_{\alpha,\kappa}(\R^N)$ and $P_\mu \in \mathfrak P_{\alpha,\mu}(\R^N)$. Then, for every $r>0$ one has
\[
\int_{S_r} P_\kappa P_\mu \frac{\psi_\alpha}{|\nabla \ra|}dH_{N-1} = 0.
\]
\end{prop}

\begin{proof}
By formula (2.30) in \cite{G} we have
\begin{align*}
0\ & = \int_{B_r} \big(P_\kappa\ \Ba P_\mu - P_\mu\ \Ba P_\kappa\big) dz dt 
\\
& = \int_{S_r} \bigg[P_\kappa<\na P_\mu,\na \ra> - P_\mu<\na P_\kappa,\na \ra>\bigg]\frac{dH_{N-1}}{|\nabla \ra|}
\end{align*}
Using the equation \eqref{nara2} in the latter identity we find
\[
<\na P_\kappa,\na \ra> = \frac{\Za P_\kappa}{\ra} \psi_\alpha =  \kappa P_\kappa \frac{\psi_\alpha}{\ra}
\]
and similarly,
\[
<\na P_\mu,\na \ra> = \frac{\Za P_\mu}{\ra} \psi_\alpha = \mu P_\mu \frac{\psi_\alpha}{\ra}.
\]
Combining the last three equations we obtain
\[
\frac{\mu - \kappa}{r}\int_{S_r} P_\kappa P_\mu \frac{\psi_\alpha}{|\nabla \ra|}dH_{N-1} = 0.
\]
Since $\kappa\not= \mu$, the desired conclusion follows.

\end{proof}

\begin{thrm}\label{T:N=k-B}
 Let $u\not\equiv 0$ be a solution of $\Ba u=0$ in $\R^N$. If for some number $\kappa\ge 0$ we have $N(u,r)\equiv \kappa$, then $u\in \mathfrak P_{\alpha,\kappa}(\R^N)$.
\end{thrm}

\begin{proof}
From Corollary \ref{C:homo} or Remark \ref{R:sh}, we conclude that $\Za u = \kappa u$ in $\R^N$, and thus $u\in \mathfrak P_{\alpha,\kappa}(\R^N)$.

\end{proof}

Our next result represents a generalization to solutions of the operator $\Ba$ of a monotonicity theorem proved in \cite{GP} for solutions of the lower-dimensional obstacle problem for the standard $\Delta$. For the case $\kappa = 2$, and in connection with solutions of the classical obstacle problem for $\Delta$, this monotonicity theorem was first proved by Monneau in \cite{M}. We begin with some preliminary considerations. Suppose $\Ba u = 0$ in $\R^N$. According to Theorem \ref{T:almgrenBa} the frequency $N(u,\cdot)$ is monotone noncreasing, and therefore the limit
\[
N(u,0^+) = \underset{r\to 0^+}{\lim}\ N(u,r),
\]
exists. If $\kappa = N(u,0^+)$, then again by Theorem \ref{T:almgrenBa} we know that
\begin{equation}\label{Nk}
N(u,r) \ge \kappa,\ \ \ \ \ 0<r<R.
\end{equation}

\begin{thrm}\label{T:MBa}
Let $\Ba u = 0$ in $B_R$ and denote by $\kappa = N(u,0^+)$. Let $P_\kappa\in \Gamma^2_\alpha(B_R)$ be such that $\Ba P_\kappa = 0$ and $\Za P_\kappa = \kappa P_\kappa$ in $B_R$,  and consider the functional
\begin{equation}\label{M}
\mathcal M_\kappa(u,P_\kappa,r) = \frac{1}{r^{Q-1+2\kappa}} \int_{S_r} (u - P_\kappa)^2 \frac{\psi_\alpha}{|\nabla \ra|}dH_{N-1}.
\end{equation}
Then, 
\begin{equation}\label{M'}
\frac{d}{dr} \mathcal M_\kappa(u,P_\kappa,r) = \frac{2}{r} \mathcal W_\kappa(u,r),
\end{equation}
and therefore by \eqref{WkBa} and \eqref{Nk}, $r \to  \mathcal M_\kappa(u,P_\kappa,r)$ is non-decreasing in $(0,R)$.
\end{thrm}

\begin{proof}
We begin by observing that, thanks to \eqref{Nk}, we have
\begin{equation}\label{M1}
N(u,r) \ge  \kappa,\ \ \ \ \ r>0.
\end{equation}
By formulas \eqref{M1} and \eqref{WkBa} we thus find
\begin{equation}\label{M2}
\mathcal W_\kappa(u,r) \ge 0,\ \ \ \ \ r>0.
\end{equation}
Therefore, the nondecreasing character of $r\to \mathcal M_\kappa(u,P_\kappa,r)$ will follow once we establish formula \eqref{M'}. We turn to this objective now.

Corollary \ref{C:homo}
 guarantees that $N(P_\kappa,r) \equiv \kappa$, and so, again by \eqref{WkBa}, we conclude that
\[
\mathcal W_\kappa(P_\kappa,r) \equiv 0.
\]
This observation allows to write, with $w = u - P_\kappa$,
\begin{align}\label{M3}
\mathcal W_\kappa(u,r) & = \mathcal W_\kappa(u,r) - \mathcal W_\kappa(P_\kappa,r) 
\\
& = \frac{1}{r^{Q-2+2\kappa}} \int_{B_r} (|\na w|^2 + 2 <\na w,\na P_\kappa>)dz dt
\notag\\
& - \frac{\kappa}{r^{Q-1+2\kappa}} \int_{S_r} (w^2 + 2 w P_\kappa) \frac{\psi_\alpha}{|\nabla \ra|}dH_{N-1}.
\notag\end{align}
Next, we integrate by parts in the term
\begin{align*}
\int_{B_r} 2 <\na w,\na P_\kappa> dzdt & = 2 \int_{S_r} w \sum_{j=1}^m <X_j,\nu> X_j P_\kappa dH_{N-1}
\\
& - 2 \int_{B_r} w \sum_{j=1}^m \text{div}(X_j P_\kappa\ X_j) dzdt
\\
& = 2 \int_{S_r} w <\na P_\kappa,\na \ra> \frac{dH_{N-1}}{|\nabla \ra|}  - 2 \int_{B_r} w \Ba P_\kappa dzdt
\\
& = 2 \int_{S_r} w <\na P_\kappa,\na \ra> \frac{dH_{N-1}}{|\nabla \ra|},
\end{align*}
since $\Ba P_k = 0$, and from \eqref{XBa} one has div$ X_j = 0$. We now use the crucial identity \eqref{nara2} in Lemma \ref{l:gradalpharho} and the hypothesis $\Za P_\kappa = \kappa P_\kappa$, to conclude that
\begin{equation}\label{keystep}
\int_{B_r} 2 <\na w,\na P_\kappa> dzdt = \frac{2\kappa}{r} \int_{S_r} w P_\kappa \frac{\psi_\alpha}{|\nabla \ra|}dH_{N-1}.
\end{equation}
We cannot emphasize enough the key role of the identity \eqref{keystep}. Substituting \eqref{keystep} in \eqref{M3}, and using the definition \eqref{WBa}
 of $\mathcal W_\kappa(w,r)$, we conclude that the following noteworthy identity holds
 \begin{equation}\label{WW}
 \mathcal W(u,r) = \mathcal W(w,r),\ \ \ \ \ 0<r<R.
 \end{equation}
We now observe that \eqref{M} implies
\begin{equation}\label{Hw}
\mathcal M_\kappa(u,P_\kappa,r) = \frac{1}{r^{Q-1+2\kappa}} \int_{S_r} w^2 \frac{\psi_\alpha}{|\nabla \ra|}dH_{N-1} = \frac{1}{r^{Q-1+2\kappa}} H(w,r).
\end{equation}
Since $\Ba w = 0$, Lemma \ref{l:hprime} gives
\[
H'(w,r) = \frac{Q-1}{r} H(w,r) + 2 D(w,r).
\]
Differentiating \eqref{Hw} and using the latter identity we thus find
\begin{align*}
\frac{d}{dr} \mathcal M_\kappa(u,P_\kappa,r) & = - \frac{Q-1+2\kappa}{r^{Q+2\kappa}} H(w,r)
 +  \frac{1}{r^{Q-1+2\kappa}}\left[\frac{Q-1}{r} H(w,r) + 2 D(w,r)\right]
 \\
 & = \frac 2r \mathcal W_\kappa(w,r) =  \frac 2r \mathcal W_\kappa(u,r),
\end{align*}
where in last equality we have used \eqref{WW}. We have established \eqref{M'}, thus completing the proof of the theorem.

\end{proof}

We close by remarking that Theorem \ref{T:MBa} has a counterpart for Carnot groups if one assumes that both the harmonic function $u$ and the homogeneous harmonic function $P_\kappa$ have vanishing discrepancy. We omit the relevant statement.

\end{document}